\documentclass[oneside]{amsart}
\usepackage{graphicx} 
\usepackage{float}
\usepackage{amsfonts}
\usepackage{tikz}
\usepackage {amssymb}
\usepackage {amsmath}
\usepackage {amsmath}
\usepackage {amsthm}
\usepackage{graphicx}
\usepackage{multirow}
\usepackage{xypic}
\usepackage {amscd}
\usepackage{mathrsfs}
\usepackage[colorlinks, linkcolor=blue, anchorcolor=black, citecolor=red]{hyperref}
\usepackage{enumerate}
\usepackage{cleveref}
\usepackage{hyperref}
\usepackage{comment}
\usepackage{geometry}  

\newcommand{\bC}{\mathbb{C}}
\newcommand{\bQ}{\mathbb{Q}}
\newcommand{\bZ}{\mathbb{Z}}
\newcommand{\bR}{\mathbb{R}}

\newcommand{\la}{\langle}
\newcommand{\ra}{\rangle}

\newcommand{\bP}{\mathbb{P}}
\newcommand{\cY}{\mathcal{Y}}
\newcommand{\cO}{\mathcal{O}}

\newcommand{\cX}{\mathcal{X}}

\newcommand{\orb}{\mathrm{orb}}
\newcommand{\cC}{\mathcal{C}}
\newcommand{\cCo}{\mathcal{C}^\circ}
\newcommand{\oYe}{\overline{U}}
\newcommand{\hmi}{\mathrm{hmi}}
\newcommand{\lSFT}{\mathrm{lSFT}}
\newcommand{\CL}{\mathcal{L}}
\newcommand{\RS}{\mathrm{RS}}
\newcommand{\LCZ}{\mathrm{LCZ}}
\newcommand{\CZ}{\mathrm{CZ}}
\newcommand{\rd}{\mathrm{d}}
\newcommand{\Id}{\mathrm{id}}
\newcommand{\Sp}{\mathrm{Sp}}
\newcommand{\fkr}{{\mathfrak{r}}}
\newcommand{\RV}{{\mathfrak{v}}}

\newtheorem{thm}{Theorem}[section]
\newtheorem{prop}[thm]{Proposition}

\newtheorem{defn}[thm]{Definition}

\newtheorem{rem}[thm]{Remark}
\newtheorem{conj}[thm]{Conjecture}

\newtheorem{exmp}[thm]{Example}
\newtheorem{lem}[thm]{Lemma}

\newtheorem{defn-prop}[thm]{Definition-Proposition}

\makeatletter
\def\namedlabel#1#2{\begingroup
    #2%
    \def\@currentlabel{#2}%
    \phantomsection\label{#1}\endgroup
}
\makeatother

\newcommand{\Addresses}{{
		\bigskip
		\footnotesize

  	Chi Li, \par\nopagebreak
        \textsc{Department of Mathematics, Rutgers University, Piscataway, NJ, U.S., 08854-8019.}\par\nopagebreak
         \textit{E-mail address:} \href{mailto:chi.li@rutgers.edu}{chi.li@rutgers.edu}

         \medskip

	     Zhengyi Zhou, \par\nopagebreak
        \textsc{State Key Laboratory of Mathematical Sciences, Chinese Academy of Sciences;}\par\nopagebreak
	    \textsc{Morningside Center of Mathematics, Chinese Academy of Sciences;}\par\nopagebreak
         \textsc{Academy of Mathematics and Systems Science, Chinese Academy of Sciences, China}\par\nopagebreak
		\textit{E-mail address}: \href{mailto:zhyzhou@amss.ac.cn}{zhyzhou@amss.ac.cn}

}}

\begin{document}

\title[K\"{a}hler compactification of $\bC^n$ and Reeb dynamics]{K\"{a}hler compactification of $\bC^n$ and Reeb dynamics}
\author{Chi Li, Zhengyi Zhou}
\date{}
\maketitle
\begin{abstract}
Let $X$ be a smooth complex manifold. Assume that $Y\subset X$ is a K\"{a}hler submanifold such that $X\setminus Y$ is biholomorphic to $\bC^n$. We prove that $(X, Y)$ is biholomorphic to $(\bP^n, \bP^{n-1})$. We then study certain K\"{a}hler orbifold compactifications of $\bC^n$ and, as an application, prove that on $\bC^3$ the flat metric is the only asymptotically conical Ricci-flat K\"{a}hler metric whose metric cone at infinity has a smooth link. As a key technical ingredient, we derive a new characterization of minimal discrepancy of isolated Fano cone singularities by using $S^1$-equivariant positive symplectic homology. 
\end{abstract}


\section{Introduction}

Our work is motivated by two problems in complex geometry. The first is Hirzebruch's classical problem on analytic compactification of $\bC^n$. 
Let $X$ be a compact complex manifold of complex dimension $n$. Assume that $Y$ is a subvariety of $X$ such that $X\setminus Y$ is biholomorphic to $\bC^n$. Of course, the simplest example is $(X, Y)=(\bP^n, \bP^{n-1})$ where $\bP^{n-1}\subset \bP^n$ is a linear subspace. In this case, we will just say that $(X, Y)$ is standard. 
In \cite{Hir54}, Hirzebruch asked the problem of classifying all such pairs $(X, Y)$ with the second Betti number $b_2(X)=1$. The Betti number condition is known to imply that $Y$ is irreducible. 
When $\dim X=1$ or $2$, it is known that there is only the standard example. For $n=3$, under the assumption that $X$ is projective, all such analytic compactifications of $\bC^3$ are now classified. The subvariety $Y$ in any known non-standard example is always singular. For example, one can compactify $\bC^n$ to a smooth quadric hypersurface in $\bP^{n+1}$ by adding an irreducible singular divisor (which is a projective cone over a lower-dimensional quadric) at infinity. 
It is a long-standing conjecture that if $Y$ is assumed to be smooth then $(X, Y)$ should be standard (see \cite{Ven62}, \cite{BM78}, \cite{PS91}). This was proved in \cite{BM78} when $\dim X\le 3$. 
Here we prove this conjecture under the K\"{a}hler assumption for all dimensions.

\begin{thm}\label{thm-main1}
With the above notation, if $Y$ is a K\"{a}hler submanifold, $(X, Y)$ is standard. 
\end{thm}

Early results of Brenton-Morrow from \cite{BM78} showed already that, with the K\"{a}hler assumption, $Y$ must be a Fano manifold and they reduce the problem to proving that $Y$ is biholomorphic to $\bP^{n-1}$ (see Theorem \ref{thm-BM}). We will prove that the Fano index of $Y$ must be $n$, which then implies that $Y$ is $\bP^{n-1}$ by a well-known criterion of Kobayashi-Ochiai (\cite{KO73}). 
The main tools we use are ($S^1$-equivariant, positive) symplectic homology and two ways to compute it, one from the functorial properties and the special symplectic topology of $\bC^n$, and the other from the Morse-Bott spectral sequences applied to the Boothby-Wang contact structure on the unit circle bundle of the normal bundle of $Y$. 

Another problem that motivates our work is the study of asymptotically conical (AC) Ricci-flat K\"{a}hler metrics, which goes back to the work of Tian-Yau \cite{TY91}. 
Let $(W, g)$ be a complete K\"{a}hler manifold. It is called asymptotically conical (AC) if it is asymptotic to a K\"{a}hler cone $(\cC, g_0)$ (see Definition \ref{def-AC}). By the compactification results from \cite{CH24, Li20} (see Theorem \ref{thm-compactify}), the classification of such K\"{a}hler metrics is closely related to the orbifold compactifications of $W$. We have the following folklore conjecture:
\begin{conj}\label{conj-AC}
    Let $g$ be a complete AC Ricci-flat K\"{a}hler metric on $\bC^n$ whose asymptotical cone $\cC$ has a smooth link. Then $g$ is the flat Euclidean metric. 
\end{conj}
In \cite[Theorem 5.2]{Tia06}, Tian proved that this is true when $n=2$, answering a question of Calabi, and in this case, the condition of the asymptotical metric cone $\cC$ having a smooth link automatically holds true thanks to a regularity result of Cheeger-Tian. He then asked a similar question in higher dimensions (see \cite[Remark 5.3]{Tia06}). 
By the recent examples constructed in \cite{Li19, CR21, Sze19}, when $n\ge 3$,  the condition of having a smooth link is needed. 
We will use the symplectic techniques in an orbifold setting and reduce the Conjecture \ref{conj-AC} to Shokurov's conjecture (see Conjecture \ref{conj-Shok}) for isolated cone singularities.
\begin{thm}\label{thm-AC}
The following statements are true:
\begin{enumerate}
    \item 
    Let $\cY\subset \cX$ be an orbifold K\"{a}hler compactification of $\bC^n$ that is associated to a complete AC K\"{a}hler metric (obtained from Theorem \ref{thm-compactify}). Then the orbifold cone $\cC(\cY,\CL)$ over $\cY$ with $\CL$ being the orbifold normal bundle is a Gorenstein klt singularity which has the minimal discrepancy equal to $n-1$. 
    \item Conjecture \ref{conj-AC} is true assuming Shokurov's conjecture. In particular, Conjecture \ref{conj-AC} is true when $n\le 3$.
    \end{enumerate}
    
\end{thm}

As a key ingredient for proving the above results, we derive the following new formula for the minimal discrepancy of isolated Fano cone singularities (see \S \ref{sec-Fano} for the explanation of these terminologies). Our formula is motivated by the work of McLean (\cite{McL16}), who introduced the $\mathrm{lSFT}$-invariant (see Definition \ref{def:lSFT}) as a generalized Conley-Zehnder type index for closed Reeb orbits.

\begin{thm}\label{thm:CZ}
    Let $o\in \cC$ be an isolated Fano cone singularity of dimension $n$. For any quasi-regular conic contact form $\eta$ on the contact link $M$, we have the following formula for the minimal discrepancy:
    \begin{equation}
       2\,\mathrm{md}(o, \cC)=\inf_{\gamma} \mathrm{lSFT}_{\eta}(\gamma)>-2. 
    \end{equation}
    Here $\gamma$ on the right ranges over all closed Reeb orbits of $\eta$. If moreover $M$ admits a Liouville filling $W$ such that $c_1^{\bQ}(W)=0\in H^2(W;\bQ)$, then we have
    \begin{equation}
    2\,\mathrm{md}(o, \cC)=\inf \{d\;|\; SH^{+,S^1}_d(W;\bQ)\neq 0\}+n-3
     \end{equation}
    where $SH^{+, S^1}_*(W;\bQ)$ denotes the $\bQ$-coefficient $S^1$-equivariant positive symplectic homology of the Liouville filling $W$.  
\end{thm}
In the more special case of isolated quotient singularities, this formula recovers the result in \cite[Proposition 2.12]{Zh24}. We remark that our proof of Theorem \ref{thm:CZ} does not use the main formula from \cite[Theorem 1.1]{McL16} (and in fact reproves it in Proposition \ref{prop:hmi} for our setting) which states that $2\cdot\mathrm{md}$ is equal to the highest minimal index denoted by $\mathrm{hmi}$ for canonical singularities. Indeed, the definition of the contact invariant $\mathrm{hmi}$ in \cite{McL16} (see \eqref{eq-hmi}) requires enumeration through all contact forms\footnote{The minimal discrepancy is achieved by certain contact forms constructed from a simple normal crossing resolution in \cite{McL16}}, while \Cref{thm:CZ} asserts that the minimal discrepancy can be computed using the Reeb dynamics from a single conic contact form and is witnessed by a non-trivial class in Floer homology. The latter can be approached using properties of symplectic homology, which allows us to show the equality $\mathrm{md}(o, \cC)=n-1$
for the Fano cone arising in our problems. 
The existence of a Liouville filling $W$ is just for the existence of $\bQ$-graded $S^1$-equivariant positive symplectic cohomology.
The specific choice of $W$ does not matter. In a more general setting, one can replace $SH^{+, S^1}_*(W;\bQ)$ by the linearized contact homology of $M$ w.r.t.\ a $\bQ$-graded augmentation over $\bC$, which always exists in our case (see \Cref{rmk:aug} for more discussions).

The application of tools from symplectic geometry to the compactification problem seems new to us. From a broader perspective, we use techniques from symplectic geometry to rule out certain singularities. The application of (very different) symplectic techniques to complex geometry to rule out singularities has appeared in the recent solution to the generalized Bishop problem \cite{FHYZ}.
 
\subsection*{Acknowledgments}
C. Li is partially supported by NSF (Grant No. DMS-2305296). Z.\ Zhou is supported by National Key R\&D Program of China under Grant No.2023YFA1010500, NSFC-12288201, and NSFC-12231010. C.\ Li would like to thank Prof.\ Laszlo Lempert for helpful discussions, Prof. Mircea Musta\c{t}\u{a} for comments on Shokurov's conjecture, and Prof.\ Gang Tian for his interest and useful comments. Both authors would like to thank Prof.\ Xiaojun Huang for putting them in contact. They also thank Prof.\ Thomas Peternell for communications regarding his preprint \cite{Pet24}. 

\section{Preliminaries}

\subsection{Fano cone singularities and AC K\"{a}hler metrics}
\label{sec-Fano}

Let $\cC$ be an $n$-dimensional normal affine variety with an isolated  $\bQ$-Gorenstein singularity $o\in \cC$. Here $\bQ$-Gorenstein means that the canonical divisor $K_\cC$ is $\bQ$-Cartier, i.e.\ there exists $k\in \bZ_{>0}$ such that $k K_{\cC}$ is a Cartier divisor.
We choose a log resolution of singularity $\mu: \tilde{\cC}\rightarrow \cC$ such that $\mu$ is an isomorphism away from $o\in \cC$ and $\mu^{-1}(o)$ is a simple normal crossing divisor. Then we have an identity 
\begin{equation}\label{eq-disc}
K_{\tilde{\cC}}=\mu^*K_\cC+\sum_i a(E_i, \cC) E_i
\end{equation}
where $a(E_i, \cC)\in \bQ$ is called the discrepancy of the exceptional divisor $E_i$. 
\begin{rem}\label{rem-disc}       
Analytically, one can compute discrepancies in the following way. For simplicity, we assume $k=1$ (Gorenstein case) and refer to \cite[Definition 2.1]{Kol13} for the general situation. Pick a local generator $s$ of $\cO(K_{\cC})(U)$ where $U$ is a small neighborhood of $o\in \cC$. Then $s$ is a non-vanishing holomorphic $n$-form on $U\setminus o$ and its pull-back $\mu^*s$ is in general a meromorphic $n$-form. Choose a general point on $E_i$ and nearby coordinates $\{y_1, \dots, y_n\}$ such that $E_i=\{y_1=0\}$ and $\mu^*s=f(y)dy_1\wedge \cdots \wedge dy_n$. Then by definition $a(E_i, \cC)=\mathrm{ord}_{\{y_1=0\}}f(y)$. For example, if $E$ is the exceptional divisor of the standard blowup of $ \mathbb{C}^n$ at a point, then $a(E, \mathbb{C}^n)=n-1$ . 
\end{rem}
 
\begin{defn}\label{def-klt}
   With the above notation, the singularity $(\cC, o)$ is Kawamata log terminal (klt) if $a(E_i, \cC)>-1$ for any $i$. The minimal discrepancy of $o\in \cC$ is defined as $$\mathrm{md}(o, \cC)=\min_i \left\{a(E_i, \cC); \mu(E_i)=\{o\}\right\}.$$
\end{defn}
In other words, the singularity $(\cC, o)$ is klt if and only if $\mathrm{md}(o, \cC)>-1$. 
It is known that this definition of klt singularity and the minimal discrepancy do not depend on the choice of log resolution of the singularity. Klt singularities are local analogues of Fano varieties, and they play an important role in birational algebraic geometry (see \cite{Kol13}).

\begin{conj}[Shokurov]\label{conj-Shok}
    For any klt singularity $o\in \cC$ of dimension $n$, $\mathrm{md}(o, \cC)\le n-1$ and the equality holds if and only if $o\in \cC$ is a smooth point. 
\end{conj}
Shokurov's conjecture has important applications in birational algebraic geometry, and it has been proved in dimension $\le 3$ (\cite{Mar96,Kaw93}). 

We now explain the concept of a K\"{a}hler cone metric on $\cC$. 
This means that there is smooth function $\mathfrak{r}: \cCo:=\cC\setminus o\rightarrow \bR_{> 0}$, called the radius function, which is a surjective submersion such that if we set $\hat{\omega}:=\sqrt{-1}\partial\bar{\partial}\fkr^2$ and $g_0=\hat{\omega}(\cdot, J_0 \cdot)$ where $J_0$ is the integrable complex structure on $\cCo$, then $g_0$ is a Riemannian metric on $\cCo$ that is isometric to a Riemannian cone metric $d\fkr^2+\fkr^2 g^M$ where $M=\{\fkr=1\}$ is the link and $g^M=g_0|_M$. 
Define the conic contact form associated to the radius function $\fkr$ as $\eta=-\fkr^{-1} J_0 d\fkr$. Then we have $\hat{\omega}=2\fkr d\fkr \wedge \eta+\fkr^2 d\eta$. The Reeb vector field, denoted by $\RV=J_0 (\mathfrak{r}\partial_\fkr)$, is a holomorphic Killing vector field on $\cC$ and satisfies $\eta(\RV)=1$. The associated holomorphic vector field $\RV_\bC=\fkr \partial_\fkr-\sqrt{-1} J_0(\fkr\partial_\fkr)$ generates an effective holomorphic action by a (complexified) torus $(\bC^*)^m$ which we will also denote by $\la \RV\ra$. 
If $m=1$, $\RV$ and $\eta$ are called quasi-regular and otherwise (i.e. when $m\ge 2$) $\RV$ and $\eta$ are called irregular. 

The Lie algebra of this torus $\la \RV\ra\cong (\bC^*)^m$ is isomorphic $\bC^m=\bR^m\oplus \sqrt{-1}\bR^m$ whose elements correspond to special holomorphic vector fields on $\cC$. The Reeb cone $\mathrm{t}_\bR^+$ is a convex cone in the $\bR^m$ factor whose elements correspond to the Reeb vector fields associated to radius functions as above. The Reeb cone can be described algebraically (cf. \cite[Definition 3.2]{LLX20}) and can also be considered as the dual cone to the cone image of the moment map associated with the Hamiltonian $(S^1)^m$-action on $(\cC, \hat{\omega})$ (see \cite[1.2]{MSY08}).
In general, any element $\RV'\in \mathrm{t}^+_\bR$ is called quasi-regular if $\RV'-\sqrt{-1}J_0\RV'$ generates a $\bC^*$ action and is called irregular otherwise. It is easy to see that the quasi-regular condition is equivalent to the condition that all orbits of the Reeb vector field are closed.

\begin{rem}
    As the notation conventions in the Sasaki geometry literature and contact geometry literature are different, we fix $\eta$ as the contact form, $\RV$ as the Reeb vector field, and $\xi:=\ker \eta$ as the contact structure in this paper.
\end{rem}

In the quasi-regular case, we can take the quotient of $\cCo$ by the $\bC^*$-action or, equivalently, the quotient of $M$ by the $S^1$-action, to get an orbifold $\cY:=(Y, \Delta)=M/S^1=\cCo/\bC^*$ equipped with an orbifold line bundle $\CL$. Here $Y$ denotes the quotient space and $\Delta=\sum_\alpha (1-\frac{1}{m_\alpha})D_\alpha$ denotes the branch divisor
where the sum is taken over all Weil divisors $D_\alpha$ that lie in the orbifold singular locus and $m_\alpha$ is the gcd of the orders of the local uniformizing groups taken over all points of $D_\alpha$ (see \cite[Definition 4.4.8]{BG08}). 
There is a natural projection $\pi: \cC^\circ\rightarrow Y$ which makes $\cC^\circ$ a $\bC^*$-Seifert bundle over the cyclic orbifold $\cY$ which was studied by Koll\'{a}r in \cite{Kol04}\footnote{In the literature of algebraic geometry, $\cC$ is also called an affine variety with a good $\bC^*$-action (See \cite{Kol04} and its references). }.
The affine variety $\cC$ is algebraically described as:
\begin{equation*}
    \cC=\cC(\cY, \CL):=\mathrm{Spec}\left(\bigoplus_m H^0(\cY, m \CL)\right). 
\end{equation*}
Geometrically, it is obtained from the total space of the orbifold line bundle $\CL^{-1}$ by contracting the zero section. 
For simplicity, we will call $o\in \cC$ an orbifold cone singularity. The affine variety
$\cC$ can be compactified to a projective variety $\overline{\cC}$ by adding an divisor $\cY_\infty\cong \cY$ at infinity (see \cite[14]{Kol04}). According to \cite[42]{Kol04}, the singularity $o\in \cC$ is klt if and only $(Y, \Delta)$ is a Fano orbifold (i.e. the orbifold anti-canonical line bundle $-K^\orb_{(Y,\Delta)}$ is ample) and there exists $\ell\in \bQ_{>0}$ such that $-K^\orb_{(Y,\Delta)}\sim_\bQ \ell \CL$ (see Proposition \ref{prop-klt} for a proof in our setting). In this case, we will simply call the orbifold cone $o\in \cC$ a \textit{Fano cone} (by following the terminology in \cite[Definition 3.7]{LLX20}). 
Moreover, in the quasi-regular case, the metrics $g_0$ and $g^M$ discussed in the previous paragraph are induced by a Hermitian metric $h$ on $\CL$ over $(Y, \Delta)$ whose Chern curvature will be denoted by $\omega^Y$.

We now introduce the class of K\"{a}hler metrics on quasi-projective varieties that are modeled on K\"{a}hler cone metrics near infinity. 
\begin{defn}\label{def-AC}
Let $(W, g)$ be a complete K\"{a}hler manifold whose integrable complex structure is denoted by $J$. It is called asymptotically conical (AC) with the asymptotical cone $(\cC, g_0)$ if there exists a compact subset $K\subset W$ and a diffeomorphism $\Phi: \{\fkr>1\}\rightarrow W\setminus K$ such that for some $\lambda_1>0, \lambda_2>0$ and all $j\in \mathbb{Z}_{\ge 0}$. 
\begin{eqnarray}
    &&|\nabla ^j_{g_0} (\Phi^*g-g_0)|_{g_0}=O(\fkr^{-\lambda_1-j}), \\
    &&|\nabla^j_{g_0}(\Phi^* J-J_0)|_{g_0}=O(\fkr^{-\lambda_2-j}).
\end{eqnarray}
\end{defn}

If the K\"{a}hler metric $g$ is Ricci-flat, then its asymptotical cone $(\cC, g_0)$ is also Ricci-flat. In this case, the metric $g^M$ on the link $M$ is called Sasaki-Einstein. When $\RV$ is quasi-regular, then $g^M$ is Sasaki-Einstein if and only if the orbifold K\"{a}hler metric $\omega^Y$ is K\"{a}hler-Einstein i.e.\ $Ric(\omega^Y)=\ell \omega^Y$. 
We refer to the book \cite{BG08} for these facts and an extensive exposition on Sasaki geometry, particularly Sasaki-Einstein metrics. 

\begin{exmp}[Tian-Yau metrics, \cite{TY91}]\label{exmp-TY}
    Let $X$ be a Fano manifold, which means that $X$ is a projective manifold with an ample anticanonical line bundle $-K_X$. Let $Y$ be a smooth divisor whose associated line bundle over $X$ is denoted by $[Y]$. Assume that there exists $\beta>1$ such that $-K_X=\beta [Y]$. In particular, $[Y]$ is an ample line bundle. By adjunction formula, $-K_Y=(\beta-1)L$ with $L=[Y]|_Y$ being the normal bundle of $Y\subset X$ and hence $Y$ is also a Fano manifold. Then there exists a complete AC K\"{a}hler metric on $X\setminus Y$ whose asymptotical cone $\cC$ is given by
    $\mathrm{Spec}\left(\bigoplus_{m=0}^{+\infty} H^0(Y, mL)\right)$. In fact, the K\"{a}hler form near $Y$ can be chosen to be given by $\sqrt{-1}\partial\bar{\partial}\|s\|^{-\frac{2(\beta-1)}{n}}$ where $s$ is the defining section of $Y$ and $\|\cdot\|^2$ is a smooth Hermitian metric on the line bundle $[Y]$ (\cite[(2.2)]{TY91}, see also \cite[section 5]{Li20}). 
    Moreover, if $Y$ admits a K\"{a}hler-Einstein metric with positive Ricci curvature, then there exists a complete AC Ricci-flat K\"{a}hler metric on $X\setminus Y$ by solving a complex Monge-Amp\`{e}re equation. 
    \end{exmp}

We will use the following compactification result, which only assumes that the complex structure is asymptotic to the cone. 
\begin{thm}
\label{thm-compactify}\label{thm:compactify}
Let $(W, g, J)$ be a complete K\"{a}hler manifold and $(\cC, g_0, J_0)$ be an orbifold cone as before. Assume that there exist a compact subset $K\subset W$ and a diffeomorphism $\Phi: \{\mathfrak{r}>1\}\rightarrow W\setminus K$ such that for some $\lambda_2>0$ and all $j\in \mathbb{Z}_{\ge 0}$, 
\begin{equation*}
    |\nabla^{j}_{g_0}(\Phi^*J-J_0)|_{g_0}=O(\mathfrak{r}^{-\lambda_2-j}).
\end{equation*}
\begin{enumerate}
\item (\cite{Li20, CH24}) Assume that the Reeb vector field $\RV$ is quasi-regular on the cone $\cC$. 
There exists a compact K\"{a}hler orbifold $\cX$ and a sub-orbifold divisor $\cY$ such that $\cY\cong \cC^\circ/\la \mathfrak{v}\ra$ and $\cX\setminus \cY$ is biholomorphic to $W$. Moreover there exists an orbifold diffeomorphism $\phi: (\mathcal{U}, \cY)\rightarrow (\mathcal{V}, \cY_\infty)$ where $\mathcal{U}$ is a strongly pseudo-concave neighborhood of $\cY\subset \cX$ and $\mathcal{V}$ is a strongly pseudo-concave neighborhood of $\cY_\infty\subset \overline{\cC}$ where $\overline{\cC}=\cC\cup\cY$ is the projective compactification of $\cC$ mentioned above.
\item  (\cite{CH24})
Assume that the Reeb vector field $\RV$ is irregular. Then there exists a sequence of radius functions $\mathfrak{r}_i$ on $\cC$ whose associated Reeb vector fields $\RV_i$ are quasi-regular such that $\RV_i$ converge to $\RV$ inside the Reeb cone $\mathrm{t}_\bR^+$ as $i\rightarrow+\infty$. For any $0<\epsilon\ll \lambda_2$ there exists $I$ such that for any $i\ge I$, 
$|\nabla_{g_{0,i}}^j(\Phi^*J-J_0)|=O(\mathfrak{r}_i^{-\lambda_2+\epsilon-j})$ for all $j\in \mathbb{Z}_{\ge 0}$. As a consequence, there exists an orbifold compactification of $W$ by $\cY_i=\cC^{\circ}/\la \RV_i\ra$ when $i\ge I$. 
\end{enumerate}
\end{thm} 
The compactification result in (1) was proved when $\RV$ is regular in \cite{Li20} and generalized to the orbifold setting in \cite[Appendix IV]{CH24}. One way to get the diffeomorphism $\Phi$ is to use the deformation to the normal cone, which we now briefly explain. Assume first that $\RV$ is regular so that $\cY=(Y, \emptyset)$ and $\cX=(X, \emptyset)$ in the above Theorem are smooth complex manifolds. Denote by $\mu: \mathfrak{X}\rightarrow X\times\bC$ the blowing-up of $Y\times \{0\}\subset X\times \bC$. The exceptional divisor $E$ is isomorphic to the projective bundle $\mathbb{P}(L\oplus \bC)$ where $L$ is the normal bundle of $Y$ in $X$. The divisor $E$ contains the divisor at the zero section $Y_0$ and the divisor $Y_\infty$ at infinity. 
 We then get a flat family $\pi=\mathrm{proj}_2\circ \mu: \mathfrak{X}\rightarrow \bC$ of projective varieties. Denote the fiber over $t\in \bC$ by $X_t$. The central fiber $X_0$ is the union $X\cup E $ such that $X$ and $E$ are glued along $Y\subset X$ and $Y_\infty \subset E$. On a neighborhood of $Y_0\subset \mathfrak{X}$ we can integrate a smooth vector field to construct a smooth family of diffeomorphism $\sigma_t$ from a neighborhood $\mathcal{U}$ of $Y\subset L$ to a small neighborhood $\mathcal{V}_t$ of $Y\subset X_t$ such that $\sigma_0=\mathrm{id}_{\mathcal{U}}$ and $|\sigma_t^* J-J_0|=O(|t|)$ where $J$ is the complex structure on $X_t=X$ and $J_0$ is the complex structure on $L$ (see \cite[Proposition 3.1]{Li20}). Choose $\epsilon\ll 1$ such that $M^\epsilon=\{s\in L; |s|^2=\epsilon\}$ is contained in $\mathcal{U}$. It is clear that $M^\epsilon$ is isomorphic to $M$ as CR-manifolds thanks to the holomorphic $\bC^*$ action on $L$. For $|t|\ll 1$, the CR manifolds $(M_t, \sigma_t^*J|_{M^\epsilon})$ are still strongly pseudo-convex, whose associated contact structures are constant due to Gray's stability theorem in contact geometry. More generally, if $\RV$ is quasi-regular, Conlon-Hein observed that the construction of deformation to the normal cone can be applied and moreover the restriction deformation of $\mathfrak{X}$ to a small neighborhood of $\cY\subset \cX$ is locally trivial (see \cite[Definition IV.5]{CH24}). This can be considered as a tubular neighborhood theorem in the setting of cyclic orbifolds and allows us to use the same construction on uniformizing charts, like in the regular case, to get a diffeomorphism $\Phi$ as stated in the above theorem. 

\subsection{Orbifold K\"{a}hler compactification of $\bC^n$}

First, we recall some results proved by Brenton-Morrow in \cite{BM78} (see also \cite{Ven62, PS91}).
\begin{thm}\cite[Theorem 1.1]{BM78}\label{thm-BM}
Let $X$ be a connected compact complex manifold and let $Y$ be a K\"{a}hler submanifold such that $X\setminus Y$ is biholomorphic to $\bC^n$. Then the following statements are true:
\begin{enumerate}
    \item The complex manifold $X$ is projective algebraic and $Y$ is positively embedded hypersurface. The line bundle $[Y]$ associated to the divisor $Y$ is ample over $X$. 
    \item There are ring isomorphisms $H^*(X, \bZ)\cong H^*(\bP^n, \bZ)$ and $H^*(Y, \bZ)\cong H^*(\bP^{n-1}, \bZ)$. 
    \item $H^2(X, \bZ)$ is generated by $c_1([Y])$ and $H^2(Y, \bZ)$ is generated by $c_1(L)$ where $L=[Y]|_Y$ is the normal bundle of $Y$ inside $X$. 
    \item $-K_X=(r+1)[Y]$ with $r\ge 1$ and $-K_Y=r L$. In particular, $X$ and $Y$ are both Fano manifolds. 
\end{enumerate}
\end{thm}
\begin{rem}
Rigorously speaking, in \cite{BM78}, the above results were proved in a more general situation that $X\setminus Y$ is a complex homology cell, and for that purpose, another assumption was added: $X$ contains no exceptional subvarieties. This assumption was used in the proof first to show that the normal bundle $L=[Y]|_Y\rightarrow Y$ is positive (it was shown that $L$ is either positive or negative) and then to show that $X$ is projective via a result of Grauert (\cite{Gra62}, \cite[Theorem 2.4]{MR75}). 
Here we already know that $X\setminus Y$ is $\bC^n$, which is Stein, hence it is guaranteed that the normal bundle $L$ is positive (since otherwise by Grauert's criterion (\cite{Gra62}) $Y$ can be contracted to a point $p$ to give a compact analytic space $\hat{X}$, but $X\setminus Y=\hat{X}\setminus p$ can never be Stein when the dimension $n>1$). 
\end{rem}
Now we generalize the above results to an orbifold setting. 
Let $(\bC^n, g)$ be an AC K\"{a}hler metric that is asymptotical to $(\cC, g_0)$ with a quasi-regular Reeb vector field $\RV$. 
Let $\cY\subset \cX$ be the orbifold compactification of $\bC^n$ obtained by Theorem \ref{thm-compactify}. In particular, $\cY=\cC^\circ/\la \RV\ra$. 
Let $[\cY]$ denote the orbifold line bundle on $\cX$ associated to the orbifold divisor $[\cY]$ and denote by $\CL$ its restriction to $\cY$ which coincides with the orbifold normal bundle $\CL=[\cY]|_{\cY}=N_{\cY/\cX}$. We refer to \cite[4.4]{BG08} for these notions for complex orbifolds.  We will use the orbifold cohomology introduced in \cite{Hae84}. In other words, we set $H^i_\orb(\cX, \bZ)=H^i(B\cX, \bZ)$ and $H^i_\orb(\cY, \bZ)=H^i_\orb(B\cY, \bZ)$ where $B\cX$ and $B\cY$ are the classifying spaces of the orbifolds $\cX$ and $\cY$ respectively. $B\cX$ can be constructed by the following procedure (see \cite{Hae84,HS91} and \cite[Chapter 4]{BG08} for more details). Choose a Riemannian metric on $\cX$ and let $P$ denote the differentiable bundle of orthonormal frames of $\cX$. Let $EO(2n)\rightarrow BO(2n)$ be a universal $O(2n)$-bundle, then $B\cX$ is the quotient of $EO(2n)\times P$ by the diagonal action of $O(2n)$. A similar construction applies to $\cY$. Moreover, since it is a global quotient of $M$ by an effective $S^1$ action, we can use the following construction following \cite[section 5.3]{HS91}. We choose an $S^1$-invariant Riemannian metric on $M$. Let $P_M$ denote the bundle of orthonormal frames of $M$ and set $BM=P_M\times_{O(2n-1)}EO(2n-1)$ where $EO(2n-1)\rightarrow BO(2n-1)$ is a universal $O(2n-1)$-bundle. We can set $B\cY=BM/S^1$ and we have a principal $S^1$-bundle $BM\rightarrow BM/S^1=B\cY$.
\begin{prop}\label{prop-compactify}
With the above notation, the following properties hold true:
\begin{enumerate} 
\item There are ring isomorphisms $H_\orb^*(\cX, \bQ)\cong H^*(\bP^n, \bQ)$ and $H_\orb^*(\cY, \bQ)\cong H^*(\bP^{n-1},\bQ)$. 
\item $H^2_\orb(\cX, \bZ)$ is generated by $c_1^\orb([\cY])$ and $H^2_\orb(\cY, \bZ)$ is generated by $c_1^\orb(\CL)$. 
\item $-K^\orb_\cX= (r+1) [\cY]$ with $r\ge 1$ is an ample orbifold line bundle and  $-K^\orb_\cY=r \CL$ is also ample. In particular, both $\cX$ and $\cY$ are Fano orbifolds and $(r+1)$, $r$ are their Fano indices respectively. 
\item The link $M$ of the cone $\cC$ is an integral homology sphere and bounds a homology ball in $\bC^n$. 

\end{enumerate}
\end{prop}
\begin{proof}
Let $\overline{\mathcal{U}}$ be the closed neighborhood of $\cY\subset \cX$ as obtained from Theorem \ref{thm:compactify}. We denote by $\overline{U}$ the underlying space of the orbifold $\overline{\mathcal{U}}$ (with smooth boundary). 
We have an exact sequence for ordinary cohomology:
\begin{equation}\label{eq-rel1}
H^k(X-U, \partial \oYe, \bZ)\rightarrow H^k(X-\oYe, \bZ)\rightarrow H^k(\partial \oYe, \bZ)\rightarrow H^{k+1}(X-\oYe, \partial \oYe, \bZ). 
\end{equation}
Using the excision property, we have $H^k(X-\oYe, \partial \oYe, \bZ)=H^k(X, \oYe, \bZ)=H^k(X, Y, \bZ)=H^k_c(X\setminus Y, \bZ)$ for $0\le k< 2n$. So we have 
\begin{equation}
H^k(X-\oYe, \bZ)\cong H^k(\partial \oYe, \bZ)\quad \text{ for } 0\le k\le 2n-2.
\end{equation}  On the other hand, by the Lefschetz duality, we have $H_k(X-\oYe, \bZ)=H^{2n-k}(X-\oYe, \partial \oYe, \bZ)=0$ for $1\le k\le 2n$. 
By the Universal Coefficients Theorem, $H^k(X-\oYe, \bZ)=0$ for $1\le k\le 2n$. So we also get:
\begin{equation}\label{eq-HMvanish}
H^k(\partial \oYe, \bZ)=0\quad  \text{for}\quad 1\le k\le 2n-2.
\end{equation} 

Next, we consider the exact sequence for relative orbifold cohomology:
\begin{equation}
 H^k_\orb(\cX, \cY, \bZ)\rightarrow H^k_\orb (\cX, \bZ)\rightarrow H^k_\orb (\cY, \bZ)\rightarrow H^{k+1}_\orb(\cX, \cY, \bZ)
\end{equation}
We have $H^k_\orb (\cX, \cY, \bZ)=H^k_c(\cX\setminus \cY, \bZ)=0$ for $0\le k< 2n$. Therefore we get $H^k_{\orb}(\cY, \bZ)=H^k_{\orb}(\cX, \bZ)$ for $0\le k\le 2n-2$. In particular  
$H^0_{\orb}(\cY, \bZ)=H^0_{\orb}(\cX, \bZ)\cong \bZ$ and hence $Y$ is connected.
Since $H^{2n}_{\orb}(\cX, \cY, \bZ)=H^{2n}_c(X\setminus Y, \bZ)=\bZ$ when $k=2n-1$, we get:
\begin{equation}
0\rightarrow H^{2n-1}_\orb(\cX, \bZ)\rightarrow H^{2n-1}_\orb(\cY, \bZ)\rightarrow \bZ {\rightarrow} H^{2n}_\orb(\cX, \bZ)\rightarrow H^{2n}_\orb(\cY, \bZ)\rightarrow 0.
\end{equation}
Hence for $k\ge 2n+1$, we get $H^{k}_{\orb}(\cX)\cong H^k_\orb(\cY, \bZ)$. 

Let $M$ denote the circle bundle of $\CL$, which is diffeomorphic to $\partial \oYe$. We apply the Leray spectral sequence to the $S^1$-principal bundle $BM\rightarrow B\cY$ to get the Gysin exact sequence:
\begin{equation}\label{eq-Gysin}
H^k(M, \bZ)\rightarrow H^{k-1}_\orb(\cY, \bZ)\stackrel{\alpha_{k-1}}{\longrightarrow} H^{k+1}_\orb(\cY, \bZ)\rightarrow H^{k+1}(M, \bZ)
\end{equation}
where $\alpha_{k-1}$ is the cup product with $c_1(\CL)\in H^2_\orb(\cY, \bZ)$. 
By the vanishing property \eqref{eq-HMvanish}, we know that $\alpha_{k-1}$ is an isomorphism for $1\le k\le 2n-3$. 
Since $H^0_\orb(\cY, \bZ)=\bZ$, we get $H^{p}_\orb(\cY, \bZ)=c_1^{\orb}(\CL)^{p/2}\bZ$ for $0\le p\le 2n-2$ even. 
For $k=0$ in \eqref{eq-Gysin}, we get $H^1_\orb(\cY, \bZ)=0$ and we get $H^{p}_\orb(\cY, \bZ)=0$ when $1\le p\le 2n-3$ is odd. 
By applying the same argument to $\bQ$ coefficients, we also see that the orbifold cohomology ring of $\cY$ with $\bQ$-coefficients is the same as the cohomology ring of $\bP^{n-1}$, which is generated by $c_1(\CL)$.

Next, the K\"{a}hler assumption implies that there is a Hodge decomposition of $H^k_\orb(\cY, \bC)$ and $H^k_{\orb}(\cX, \bC)$ for any $k$ (see \cite[Part I, 2.5]{PS08}). In particular, we have
\begin{equation}
H^k_\orb(\cX, \bC)=\bigoplus_{p+q=k} H^p_\orb(\cX, \Omega^q_{\cX}). 
\end{equation}
We get in particular that $H^1_\orb(\cX, \cO_{\cX})=0$ and $H^2_\orb(\cX, \cO_X)=0$. The exponential sequence in the complex orbifold setting (see \cite[Theorem 4.4.23]{BG08}) implies that the map 
$\mathrm{Pic}^\orb(\cX)=H^1_\orb(\cX, \cO^*_\cX)\cong H^2_\orb(\cX, \bZ)\cong H^2_\orb(\cY, \bZ)=\bZ$. Because $H^1(M;\bZ)=0$, we know that $c_1(\CL)$ is primitive in $H^2_\orb(\cX, \bZ)$.   

This implies in particular, $-K^\orb_\cX= b [\cY]$ for some $b\in \bZ$. If $b\le 0$, then $K^\orb_{\cX}$ has nonzero holomorphic section $s_\cY^{-b}$ which contradicts $\dim H^0_\orb(\cX, K^\orb_{\cX})=h^{n,0}_\orb(X)=0$. So $b=1+r\ge 1$ i.e. $r\ge 0$. 
By the adjunction formula, we get $-K^\orb_\cY=r \CL$. If $r=0$, then $K^\orb_\cY$ is trivial and $H^{n-1}_\orb(\cY, \cO_Y)=H^{0}_\orb(\cY, K^\orb_\cY)\cong \bC$ which contradicts to the vanishing $h^{0,n-1}(\cY)=0$. So $r\ge 1$. 

By \eqref{eq-HMvanish}, we see that $M=\partial \bar{U}$ is an integral homology sphere. By applying the Mayer-Vietoris sequence to $\bC^n=(\overline{U}\backslash \cY)\cup (\bC^n\setminus U)$, we know that $M=\overline{U}\cap (\bC^n\setminus U)$ bounds an integral homology ball $\bC^n\setminus U$.    
\end{proof}

\begin{rem}\label{rem-torsion}
    In general, $H^k_\orb(\cY, \bZ)$ is not 0 when $k\ge 2n-1$. For example, it is well-known that for the weighted projective space $\bP({\bf w})=\bP(w_1, w_2,\dots, w_n)$ with $d=\prod_i w_i$, we have the orbifold cohomology:
    \begin{equation}
       H^k_\orb(\bP({\bf w}),\bZ)= \left\{
        \begin{array}{ll}
            \bZ &  k \le 2n-2 \text{ and is even} \\
             \bZ_{d} & k\ge 2n \text{ and is even} \\
             0 & k \text{ is odd}. 
        \end{array}
        \right.
    \end{equation}
\end{rem}

We will now show that the $(2n-1)$-dimensional submanifold $M$ of $\bC^n$ from the above proposition, equipped with the induced contact structure, is cobordant to the standard contact sphere via a Liouville cobordism. 

\begin{defn}
    Given two contact manifolds $(M_+,\xi_+),(M_-,\xi_-)$, we say  $(W,\lambda)$ is a Liouville cobordism from $(M_-,\xi_-)$ to $(M_+,\xi_+)$ if 
    \begin{enumerate}
    \item $\rd \lambda$ is a symplectic form on $W$;
    \item  $\partial W= (-M_-)\cup M_+$ as oriented manifolds;
    \item The Liouville vector field $X_{\lambda}$ defined by $\lambda=\iota_{X_{\lambda}}\rd \lambda$ is transversal to $\partial W$ pointing outward/inward along $M_+/M_-$ respectively;
    \item $\ker \lambda|_{M_{\pm}}=\xi_{\pm}$.
    \end{enumerate}
    $(W,\lambda)$ is called a Liouville filling of $(M_+, \xi_+)$ if $M_-=\emptyset$. 
\end{defn}
If $M$ is the smooth boundary of strongly pseudo-convex domain $W$, then $(W,\rd^C\rho:=-Jd \rho)$ is a Liouville filling\footnote{Note that in \cite[(3.15)]{Se08}, $(\rd^C\rho)(v)=\rd \rho(Jv)=(J\rd \rho)(v)$. This explains the potential sign difference in some symplectic literature.} of $M$ whose contact structure is induced by the CR structure and $\rho$ is the strongly plurisubharmonic function whose sub-level set is $W$ with $M$ a regular level set. More generally, $(\rho^{-1}([a,b]),\rd^C\rho)$ is a Liouville cobordism from $\rho^{-1}(a)$ to $\rho^{-1}(b)$ assuming both $a,b$ are regular values. In particular, by \cite[Theorem 2.6.12]{Ho73}, we have the following.
\begin{prop}\label{prop:Liouville}
    The domain $W$ bounded by the contact link $M$ in $\bC^n$ by \Cref{prop-compactify} gives a Liouville filling of $M$.
\end{prop}
In general, being a strongly pseudo-convex hypersurface in $\bC^n$ does not imply that the cobordism from $M$ to the standard $S^{2n-1}$ end of $\bC^n$ is a Liouville cobordism. Thanks to the asymptotic cone structure, we do have a Liouville cobordism in our case, which is crucial for our computation of symplectic homology.
\begin{prop}\label{prop:cobordism}
The cobordism from the contact link $M$ to the standard $S^{2n-1}$ end of $\bC^n$ admits a Liouville cobordism structure from the CR structure on $M$ to the standard contact/CR structure on $S^{2n-1}$.
\end{prop}
\begin{proof}
    Since $M$, as a contact manifold, can be viewed as a level set of the radius function $\mathfrak{r}^2$, where $\mathfrak{r}^2$ defines a strongly plurisubharmonic function on the end of $\bC^n$ by the asymptotic cone assumption. Let $\rho=|z|^2$ using the Euclidean metric. Then $f(\rho)$ is  plurisubharmonic on $\bC^n$ if $f''\ge 0,f'\ge 0$. We assume $M$ is contained in the ball $B(r)$ for a fixed $r\gg 1$. We can choose $f$, such that $f''\ge 0,f'\ge 0$, $f(\rho)=0$ for $\rho \le r^2$. Fix $R>r$ and we require $f$ to grow sufficient fast such that so that $|\nabla \mathfrak{r}^2|<\frac{1}{2} |\nabla f(\rho)|$ outside the ball $B(R)$. As a consequence, for some $C\gg 1$, $C$ is a regular value of $F_t:=t\mathfrak{r}^2+f(\rho)$ for all $t\in [0,1]$. The function $F_1=\mathfrak{r}^2+f(\rho)$ is strongly plurisubharmonic outside $M$ and inside $F_1^{-1}(C)$, where $M$ is a regular level set of $F_1$ by construction.  Therefore, we have a Liouville cobordism from $M$ to $F_1^{-1}(C)$. Finally, as $F_t^{-1}(C)$ is a smooth family of strongly pseudo-convex hypersurfaces in $\bC^n$, the Gray stability implies that $F_1^{-1}(C)$ is isomorphic to $F_0^{-1}(C)$ as contact manifolds, while $F_0^{-1}(C)$ is the standard contact sphere by construction. 
\end{proof}

For the rest of this subsection, we study the algebraic property of the singularity $(\cC, o)$, which is again the asymptotical cone associated to a complete AC K\"{a}hler metric on $(\bC^n, J_0)$ obtained by Theorem \ref{thm-compactify}.

\begin{prop}\label{prop-klt}
    With the above notation, the vertex $o\in \cC$ is a Gorenstein klt singularity. 
\end{prop}

\begin{proof}
The proof is similar to \cite[3.1]{Kol13}. 
Let $\cC'$ denote the total space of the orbifold line bundle $\pi_{\cC'}: \CL^{-1}\rightarrow \cY$.  Then we have the identity $K_{\cC'}\otimes \cO(\cY_0)=\pi_{\cC'}^*K^{\orb}_{\cY}$. 
Away from the zero section $\cY_0$ we have $K_{\cC^\circ}=\pi^* K^{\orb}_\cY=\pi^* (\CL^{-r})\cong \cO_{\cC^\circ}$ where $\pi: \cC^\circ\rightarrow \cY$ is the natural projection. Since $\cC$ is normal, $K_{\cC}$ is also trivial and hence $\cC$ is Gorenstein. 

Let $\mu_1: \cC'\rightarrow \cC$ be the contraction of the zero section. We have the identity:
\begin{equation*}
    K_{\cC'}=\mu_1^*K_{\cC}+(r-1)\cY_0
\end{equation*}
where $\cY_0$ is the zero section of $\CL^{-1}$. Let $\mu_2: \tilde{\cC}\rightarrow \cC'$ be the resolution of singularities with simple normal crossing exceptional divisors $\{F_i\}$, we get an identity similar to \eqref{eq-disc}: 
\begin{eqnarray*}
    K_{\tilde{\cC}}+\widetilde{\cY}_0&=&\mu_2^*(K_{\cC'}+\cY_0)+\sum_i a_i F_i\\
    &=&\mu_2^*\mu_1^*K_{\cC}+r \mu_2^*\cY_0+\sum_i a_i F_i\\
    &=&\mu^* K_{\cC}+r\widetilde{\cY}_0+\sum_i (a_i+r\cdot \mathrm{ord}_{F_i}(\mathcal{O}(-\cY_0))F_i. 
\end{eqnarray*}
Here $\cO(-\cY_0)$ denotes the ideal sheaf of $\cY_0\subset \cC'$ and $\widetilde{\cY}_0$ is the strict transform of $\cY_0$ under $\mu_2$ and $\mu=\mu_2\circ \mu_1: \tilde{\cC}\rightarrow \cC$ is a resolution of singularity. The coefficient $a_i=a(F, \cC', \cY_0)$ defined by the first equality is called the discrepancy of the exceptional divisor $F_i$ with respect to the pair $(\cC', \cY_0)$ (see \cite[Definition 2.4]{Kol13}). 
Note that $\cY_0$ has only quotient singularities which are klt. By the theorem of inversion of adjunction (see \cite[Theorem 4.9]{Kol13}), $(\cC', \cY_0)$ is plt (meaning ``purely log terminal", see \cite[Definition 2.8]{Kol13}) which implies $a_i=a(F_i, \cC', \cY_0)>-1$. 
So we get the formula:
\begin{equation*}
    \mathrm{md}(o\in \cC)=\min\{r-1, a_i+r \cdot \mathrm{ord}_{F_i}(\cO(-\cY_0))\}>-1
\end{equation*}
which implies that $o\in \cC$ is indeed klt (see Definition \ref{def-klt}). 
\end{proof}
By the linearity property of discrepancies (see \cite[Lemma 2.5]{Kol13}), $a(F_i, \cC', (1-r)\cY_0)=a(F_i, \cC', \cY_0)+r\cdot \mathrm{ord}_{F_i}(\cO(-\cY_0))$. Therefore, we can also write the minimal discrepancy of $o\in \cC$ as
\begin{equation}\label{eq-md1}
    \text{md}(o\in \cC)=\min\{ r-1, a(F_i, \cC', (1-r)\cY_0) \}. 
\end{equation}
We can then derive an algebraic formula for the minimal discrepancy of $o\in \cC$ by adapting the argument as in \cite[Proof of Theorem 3.21]{Kol13}. 
First note that $\cC'$ has only quotient singularities that are contained in $\cY_0$. Let $p\in \cC'$ be a quotient singularity given by $\pi_p: \bC^n\rightarrow \bC^n/G_p$ where $G_p$ is a cyclic subgroup of $U(n)$. According to the local classification of smooth Seifert $\bC^*$-bundles from \cite[25]{Kol04}, $\bC^n/G_p$ is isomorphic to $\bC^n/\frac{1}{m}(b_1,\dots, b_{n})$ that satisfies the condition that $\pi_p^*\cY_0=\{x_1=0\}$ with $\{x_1,\dots, x_n\}$ being linear coordinates on $\bC^n$ and $b_1$ is relatively prime to $m$. 
We have the following formula for the minimal discrepancy of $o\in \cC$, which can be seen as a generalization of Reid-Tai's criterion for klt quotient singularities to the case of isolated Fano cone singularities. 
\begin{prop}
With the above notations, we have the equality
\begin{equation}\label{eqn:md}
\mathrm{md}(o, \cC)=\min_{p,g}\left\{r, \frac{1}{m}\left(r w_1(g)+\sum_{i=2}^n w_i(g)\right)\right\}-1
\end{equation}
where $p$ ranges over all quotient singularities on $\cC'$ and $g$ ranges over all non-identity elements in the cyclic groups $G_p\cong \bZ_m$ and satisfies $g^*x_i=\epsilon^{w_i(g)} x_i$ with $\epsilon=e^{2\pi \sqrt{-1}/m}$ and $0\le w_i(g)<m$. 
\end{prop}
\begin{proof}
We pick any exceptional divisor $F$ over $\bC^n/G_p$ dominated by an exceptional divisor $E$ over $\bC^n$ that is pointwise fixed by a cyclic subgroup $\la g\ra$ of order $k>1$. Assume $m=k\ell$. We have the equality
$a(F, \bC^n/G_p, (1-r)\cY_0)+1=k^{-1} (a(E, \bC^n, \Delta')+1)$ where $\Delta'=(1-r)\pi_p^*\cY_0$ (see \cite[(2.42.4)]{Kol13}).  
At a general point of $E$ contained in a birational model over $X$, choose local coordinates $y_1, \dots, y_n$ such that $E=\{y_1=0\}$, $g^*y_1=\epsilon^\ell y_1$ is a $g$-eigenfunction (up to a change of the generator of $\la g\ra$). 
Using the linear coordinates $\{x_i\}$ on $\bC^n$ as above such that $\pi_p^*\cY_0=\{x_1=0\}$, we can write $g^* x_i=y_1^{c_i}u_i$ where the $u_i$ are units. Thus (the pullback of) $g^* (dx_1\wedge \cdots \wedge dx_n)$ vanishes to order equal to $(-1+\sum_i c_i)$ along $E$. By the definition of discrepancy (see remark \ref{rem-disc}), we get $a(E, \bC^n)=\sum_i c_i-1$, and by the linearity property $a(E, \bC^n, \Delta')+1=a(E, \bC^n)+1+(r-1) c_1\ge \sum_i c_i+(r-1)c_1$. 
 If $g^* x_i =\epsilon^{w_i} x_i$ with $0\le w_i<m$, then $\ell c_i\equiv w_i \text{ mod } m$. So we get:
\begin{eqnarray*}
 a(F, \bC^n/G_p, (1-r)\cY_0)+1&=&\frac{1}{k}(a(E, \bC^n,\Delta')+1) \\
 &\ge& \frac{1}{k}(\sum_i c_i+(r-1) c_1) \ge \frac{1}{k\ell}(\sum_{i=1}^n w_i+(r-1)w_1)\\
 &=&\frac{1}{m}\left(r w_1(g)+\sum_{i=2}^n w_i(g)\right).
\end{eqnarray*}
Because of \eqref{eq-md1}, this proves that the left-hand side of \eqref{eqn:md} is greater than the right-hand side. 

Conversely any element $g\in G_p$ generates a cyclic subgroup which acts faithfully on $\bC^n$ as $\bC^n/\frac{1}{m}(w_1,\dots, w_n)$ where $\cY_0=\{x_1=0\}$. We want to show that $\frac{1}{m}(\sum_i w_i+(r-1) w_1)\ge \mathrm{md}(o, \cC)+1$.  
If $\mathrm{gcd}(w_1,\dots, w_n)=\ell>1$, then the element $\frac{1}{m}(w_1/\ell, \cdots, w_n/\ell)$ contributes a smaller number to the right-hand-side. So we assume that $\mathrm{gcd}(w_1,\dots, w_n)=1$. Consider the weighted blow-up $\bC^n$ with weights $(w_1, \dots, w_n)$. By \cite[pg. 106]{Kol13}, a local chart is given by 
\begin{equation*}
    f: \bC^n_{\mathbf{z}}/\frac{1}{w_1}(1, -w_2,\dots, -w_n)\rightarrow \bC^n_{\mathbf{x}} 
\end{equation*}
with $f^*x_1=z_1^{w_1}$, $f^*x_i=z_1^{w_i} z_i$ for $i\ge 2$. 
Let $E$ denote the exceptional divisor of this weighted blowup that dominates an exceptional divisor $F$ over $\bC^n/G$. 
Because
$f^* (x_1^{r-1} dx_1\wedge \cdots\wedge dx_n)=w_1 z_1^{-1+(r-1)w_1+\sum_{i=1}^n w_i}dz_1\wedge\cdots \wedge dz_n$ and the covering $f$ is generically unramified along the exceptional divisor $\{z_1=0\}$, we get
\begin{equation*}
    a(F, \bC^n/G_p, (1-r)\cY_0)+1=\frac{a(\{z_1=0\},\bC^n, (1-r) \{z_1=0\})+1}{m}=\frac{1}{m}((r-1)w_1+\sum_{i=1}^n w_i). 
\end{equation*}
By \eqref{eq-md1} we get the formula \eqref{eqn:md}.
\end{proof}

\begin{exmp}
Set $(\cC, o)=(\bC^n, 0)$ and let $\bC^*$ act on $\bC^n$ by $t\cdot (z_1, \dots, z_n)=(t^{a_1}z_1, \dots, t^{a_n} z_n)$ with $a_i\in \bZ_{>0}$, $ a_1\ge \dots \ge a_n $ and $\gcd(a_1,\dots, a_n)=1$. 
The corresponding orbifold $\cY=(\bC^n\setminus 0)/\bC^*=\bP(a_1, \dots, a_n)$ is the weighted projective space with the Fano index $r=\sum_{j=1}^n a_j$ which is always greater than $n$ except for $\bP^{n-1}=\bP(1,\dots, 1)$. It is easy to verify that the minimum on the right-hand-side of \eqref{eqn:md} is obtained when $p=[1, 0,\dots, 0]$ and $g=\epsilon=e^{2\pi\sqrt{-1}/a_1}\in G_p=\la \epsilon\ra\cong \bZ_{a_1}$. Indeed, we have $w_1(g)=1$, $w_i(g)=a_1-a_i$ for $2\le i\le n$ and hence
\begin{eqnarray*}
    \frac{1}{m}\left(r w_1(g)+\sum_{i=2}^n w_i(g)\right)&=&\frac{1}{a_1}\left(\sum_{j=1}^n a_j+\sum_{i=2}^n (a_1-a_i)\right)\\
    &=&n=\mathrm{md}(0, \bC^n)+1
\end{eqnarray*}
\end{exmp}

    \subsection{Conley-Zehnder indices and SFT degrees}\label{ss:CZ}
For later purposes, we briefly review the basics of Conley-Zehnder indices and related concepts. Given a path $\phi:[0,1] \to \Sp(2n)$ in the symplectic matrices, such that $\phi(0)=\Id$ and $\phi(1)\in \Sp^*(2n):=\left\{A\in \Sp(2n)\left| \det(A-\Id)\ne 0\right.\right\}$, we can define the integral Conley-Zehnder index $\mu_{\CZ}(\phi)$, which is characterized by the following properties \cite[Proposition 8]{Gutt14}:
\begin{itemize}
    \item[(P1)] Given any path $\psi:[0,1]\to \Sp(2n)$, we have $\mu_{\CZ}(\psi \phi \psi^{-1})=\mu_{\CZ}(\phi)$;
    \item[(\namedlabel{p:homotopy}{P2})]The Conley-Zehnder index is homotopy invariant for homotopies of paths from $\Id$ to $\Sp^*(2n)$;
    \item[(\namedlabel{p:sum}{P3})] Given paths $\phi_1,\phi_2$ in $\Sp(2n_1),\Sp(2n_2)$ respectively, then the path $\phi_1\oplus \phi_2$ in $\Sp(2n_1+2n_2)$ has Conley-Zehnder index $\mu_{\CZ}(\phi_1)+\mu_{\CZ}(\phi_2)$;
    \item[(\namedlabel{p:loop}{P4})] Given a loop $\psi:[0,1]\to \Sp(2n)$ with $\psi(0)=\psi(1)=\Id$, then 
    $$\mu_{\CZ}(\phi\psi)=\mu_{\CZ}(\phi)+2\mu(\psi)$$
    where $\mu(\psi)$ is the Maslov index of the loop $\psi$, i.e.\ the degree of $S^1\to \Sp(2n)\to U(n)\stackrel{\det_{\bC}}{\longrightarrow} U(1)$, here $\Sp(2n)\to U(n)$ is the homotopy inverse of the inclusion $U(n)\to \Sp(2n)$; 
    \item[(\namedlabel{p:det}{P5})] $(-1)^{n-\mu_{\CZ}(\phi)} = \mathrm{sign}\det(\Id-\phi(1))$;
    \item[(\namedlabel{p:sign}{P6})] Let $S$ be a symmetric matrix such that $|S|<2\pi$, for $\phi(t)=\exp(tJS)$ where $J$ is the standard complex structure on $\bC^n\simeq \bR^{2n}$, we have $\mu_{\CZ}(\phi)$ is half of the signature of $S$.
\end{itemize}
Properties \eqref{p:homotopy}, \eqref{p:loop} and \eqref{p:sign} determine Conley-Zehnder index uniquely, see \cite[Proposition 9]{Gutt14}.

The Conley-Zehnder indices have the following generalizations.
\begin{itemize}
    \item[(\namedlabel{p:RS}{P7})] Conley-Zehnder index has a generalization called the Robbin-Salamon index $\mu_{\RS} \in \frac{1}{2}\bZ$ defined for any path in $\Sp(2n)$. One important feature of $\mu_{\RS}$ is that it is additive w.r.t.\ concatenation of paths \cite[Property 4 of Lemma 26]{Gutt14}. By property \eqref{p:loop} above, we have $\mu_{\RS}(\psi)=2\mu(\psi)$ for a loop $\psi$. Moreover, we have the following normalization formula for the path in $U(1)\subset Sp(2)$:
    \begin{equation}\label{eq-RSnorm}
    \mu_{\RS}(\{e^{i\theta}; 0\le \theta\le T\})=
    \left\{
        \begin{array}{ll}
           \frac{T}{\pi}, & \text{if } T\in 2\pi \bZ \\ 
           2 \lfloor \frac{T}{2\pi}\rfloor+1,  & \text{if } T\not\in 2\pi \bZ.
        \end{array}
        \right.
    \end{equation}
    \item[(\namedlabel{p:RS2}{P8})] $\mu_{\CZ}$ has a lower semi-continuous extension $\mu_{\LCZ}\in \bZ$ to all paths in $\Sp(2n)$ starting from $\Id$ \cite[Definition 2.11]{McL16}. By \cite[Lemma 4.10]{McL16}, \
    \begin{equation}\label{eqn:LCZ}
        \mu_{\LCZ}(\phi)=\mu_{\RS}(\phi)-\frac{1}{2}\dim \ker (\phi(1)-\Id).
    \end{equation}
\end{itemize}
We now turn to the geometric setting. For a contact manifold $(M^{2n-1},\eta)$, let $\phi_t$ denote the flow generated by the Reeb vector field $\RV$.
For a closed Reeb orbit $\gamma$, after choosing a trivialization of $\gamma^*(\xi:=\ker \eta)$, the linearized flow gives rise to a path in $\Sp(2n-2)$. By properties \eqref{p:homotopy} and \eqref{p:loop}, the Conley-Zehnder index (including $\mu_{\LCZ},\mu_{\RS}$) only depends on the induced trivialization $\tau$ of $\det_{\bC}\xi$, which will be denoted by $\mu_{\CZ}^{\tau}(\gamma)$ ($\mu^{\tau}_{\LCZ}(\gamma),\mu^{\tau}_{\RS}(\gamma)$) respectively. To get Conley-Zehnder indices for all Reeb orbits, we have the following:
\begin{enumerate}[(i)] 
    \item For any Reeb orbit $\gamma$, $\mu^{\bZ_2}_{\LCZ}(\gamma)\in {\bZ_2}$ is well-defined. This is defined by $\mu_{\LCZ}^{\tau}(\gamma) \mod 2$ for any trivialization $\tau$ of $\det_{\bC}\gamma^*\xi$. The well-definedness follows from Property \eqref{p:loop}.
    \item If $c_1(\xi) =0$, $\det_{\bC} \xi$ can be trivialized over $M$ (the set of trivializations is parameterized by $H^1(M;\bZ)$). After choosing a trivialization $\tau$ of  $\det_{\bC} \xi$,  we get $\mu^{\tau}_{\LCZ}(\gamma)$ for all Reeb orbits and it is independent of $\tau$ if $[\gamma]=0\in H_1(M;\bQ)$. In the case, $\mu^{\bZ_2}_{\LCZ}(\gamma)=\mu_{\LCZ}(\gamma) \mod 2$.
    \item If $c^{\bQ}_1(\xi) =0$ in $H^2(M;\bQ)$, we can trivialize $\det_{\bC} \oplus^N \xi$ for some $N\in \mathbb{Z}_{>0}$, after choosing a trivialization $\tau$, this allows us to define $\mu_{\LCZ}^{\bQ,\tau}(\gamma)=\mu_{\LCZ}^{\tau}(\oplus^N \phi)/N\in \bQ$, where $\phi$ is the linearized flow along $\gamma$. This definition is independent of $N$ and is moreover independent of $\tau$ if $[\gamma]=0\in H_1(M;\bQ)$ \cite[\S 3.3.2]{gironella2021exact}. This rational index coincides with the integer index above when $c_1(\xi)=0$. In general, one should think $(\mu^{\bQ}_{\LCZ},\mu^{\bZ_2}_{\LCZ})\in \bQ\times \bZ_2$ is a pair of indices associated to an orbit. 
\end{enumerate}
\begin{defn}\label{def:lSFT}
    Let $(M, \xi)$ be a $2n-1$ dimensional contact manifold with rational first Chern class $c^{\bQ}_1(\xi)=0$ and $H^1(M;\bQ)=0$. For any Reeb orbit $\gamma$ for a fixed contact form $\eta$, we define its lower SFT degree 
    $$\lSFT_\eta(\gamma):=\mu^{\bQ}_{\LCZ}(\gamma)+n-3.$$ 
\end{defn}
Under the same assumption, McLean\footnote{The Robbin-Salamon index was denoted as $\CZ(\gamma)$ in \cite{McL16}} \cite{McL16} defined the highest minimal index by
\begin{equation}\label{eq-hmi}\hmi(M):=\sup_{\eta}\inf_{\gamma} \left\{\lSFT_\eta(\gamma)\right\}\in \bR \cup \{\pm \infty\},
\end{equation}
where $\eta$ ranges over all contact forms and $\gamma$ ranges over all Reeb orbits of $\eta$. When $M$ is the contact link of an isolated $\bQ$-Gorenstein normal singularity, McLean \cite[Theorem 1.1]{McL16} proved that $\hmi(M)$ is twice the minimal discrepancy number if $\hmi(M)\ge 0$ and  $\hmi(M)<0$ is equivalent to that the minimal discrepancy number is negative.

\section{Proof of main results}
In the following, we study the Fano index of the orbifold $\cY$ as well as the minimal discrepancy number of the cone singularity $\cC(Y,\CL)$ using Floer cohomology from symplectic geometry to finish the proof of the main theorem.

\subsection{Quasi-regular Sasaki manifolds}
Returning to the quasi-regular Sasaki case, the contact manifold $M$ has a contact form $\eta$ whose Reeb vector field $\RV$ generates an $S^1$-action and $\cY=M/S^1$ is a symplectic orbifold. The standing topological assumption in this subsection is that 
$$c^{\bQ}_1(\xi)=0.$$ This is a consequence of the cone singularity being $\bQ$-Gorenstein and
guarantees $\bQ$-valued indices ($\mu_{\LCZ},\mu_{\RS}$) after choosing a trivialization of $\det_{\bC}\oplus^N \xi$. In fact, by \cite[Lemma 3.3]{McL16} and \cite[Corollary 5.17]{BdFFU}, it is equivalent to the singularity being $\bQ$-Gorenstein given that the singularity is klt. In the special case here, we do not need to assume $H^1(M;\bQ)=0$, as all Reeb orbits of a conic contact form are rationally null-homologous, making the index independent of the trivialization. Moreover, in the question of compactification of $\bC^n$, we have $H^2(M;\bZ)=0$ by \Cref{prop-compactify} guaranteeing $\bZ$-valued indices. Let $N$ be a positive integer such that $Nc_1(\xi)=0$. Note that for any $x\in M$, the direct sum of $N$ copies of the linearization of the $S^1$-action along the orbit $S^1\cdot x$ gives rise to a loop in $\Sp((2n-2)N)$ by trivializing $\oplus^N\xi \subset TM$ along the orbit, where the trivialization matches with the unique (by our topological assumptions) global trivialization of $\det_{\bC} \oplus^N \xi$. We use $R(M)$ to denote the Maslov index $\mu$ of this loop divided by $N$, which is independent of the point $x$ by the homotopy invariant property of the Maslov index.

\begin{lem}\label{lem:rational_change}
    Let $(M,\xi)$ be a contact manifold with $c_1^{\bQ}(\xi)=0$ and $W$ a symplectic filling. For a rationally null-homologous Reeb orbit $\gamma$, and a bounding surface $u:S\to W$ such that $\partial u = \gamma$, using the unique trivialization of $u^*\det_{\bC}TW$ over $S$, we get an integer lower Conley-Zehnder index $\mu^u_{\LCZ}(\gamma)$. The rational lower Conley-Zehnder index is given by $$\mu^{\bQ}_{\LCZ}(\gamma)=\mu_{\LCZ}^u(\gamma)-2\la c_1(W), u_*[S]\ra,$$ 
    where we view $c_1(W)$ as in $H^2(W,M;\bQ)$ as $c_1^{\bQ}(\xi)=0$ and $u_*[S]\in H_2(W,M;\bZ)$. Alternatively, if $A\in H_2(W;\bQ)$ satisfy $[A]=u_*[S]$ in $H_2(W,M;\bQ)$ (such $A$ exists as $u_*[\partial S]=0\in H_1(M;\bQ)$), then the rational lower Conley-Zehnder index $\mu^{\bQ}_{\LCZ}(\gamma)=\mu_{\LCZ}^u(\gamma)-2\la c_1(W),A \ra$.
\end{lem}
\begin{proof}
    Assume $Nc_1(\xi)=0$, and write $\phi(t)$ as the linearized flow of the Reeb flow. We then apply Property \eqref{p:loop} to $\oplus^N \phi(t)$ for trivializations from $u$ and the global trivialization of $\det_{\bC}\oplus^N \xi$, which are differed by $2N\la c_1(W),u_*[S] \ra$.
\end{proof}

Under the above setup, by the $\bQ$-Gorenstein property or $c_1^{\bQ}(\xi)=0$, there is $r(\cY, \CL)\in \bQ$ such that 
$-K_{\cY}=r(\cY,\CL)\CL$. 
In the case of \Cref{prop-compactify}, $r(\cY,\CL)$ is the Fano index of $\cY$.

\begin{prop}\label{prop:Fano_Maslov}
    Under the setup above, we have the identity $R(M)=r(\cY,\CL)$. 
\end{prop}
\begin{proof}
    We can compute the Maslov index by the induced $S^1$-symplectic action on the symplectization $\bR_s\times M$, where the symplectic tangent bundle along an orbit is the contact structure direct sum with $\langle \RV, \partial_s\rangle$. Along with the obvious trivialization of $\langle \RV, \partial_s\rangle$, the linearized action along the orbit has an extra block of $2\times 2$ identity matrix compared to the contact case. In particular, the Maslov index does not change. If we compute the Maslov index using the induced trivialization from the bounding disk $D$--the fiber disk--in the orbifold line bundle $\CL^{-1}$, then the loop of the symplectic matrices is just a full rotation in the two coordinates in the disk $D$, hence the Maslov index is $1$. As $2R(M)-n+1$ is $\mu^{\bQ}_{\LCZ}$ of the principal Reeb orbit (the $S^1$ orbit that is not multiply covered) on $M$, $R(M)$ enjoys the same property as $\mu^{\bQ}_{\LCZ}$ in \Cref{lem:rational_change} without the factor $2$. Let $A$ be a closed rational chain whose intersection number with $\cY$ is $1$, we may view $A$ as contained in the zero section in $\CL^{-1}$. Then we have $[A]=[D]$ in $H_2(\CL^{-1},M;\bQ)$ and $\langle c_1(\CL^{-1}), A \rangle = 1$. We have $R(M)=1-\la c_1(Tot(\CL^{-1})), A \ra= 1-\langle c_1(\cY)+c_1(\CL^{-1}), A \rangle=r(\cY,\CL)$, where $Tot(\CL^{-1})$ is the total space of $\CL^{-1}$ as a complex orbifold.
\end{proof}

We consider the finite set $S$ of isotropy groups (including the trivial group, which is the isotropy group for a generic point) of the $S^1$ action on $M$. The set $S$ is equipped with a partial order, we say $G_x>G_y$ if $G_y\subset G_x\subset S^1$ is a subgroup. For $G\in S$,  the quotient of the fixed point set $M^G/S^1$ gives rise to a branch locus $\cY_G$ of the quotient K\"ahler orbifold $\cY=M/\langle \RV \rangle$ giving $M$ a stratification over the partial order set $S$. For non-minimal $G\in S$, we use $G^-$ to denote the unique maximal element that is smaller than $G$. We formally define $G^-=\emptyset$ when $G$ is the minimal element of $S$. We may rescale the contact form, such that the principal orbit (the simple Reeb orbit over a non-singular point of $M/S^1$) has a period $1$. For $G\in S$, we write $\cY_G=\sqcup_{i\in C_G}\cY^i_G$ as the decomposition into disconnected components. $M^G$ has a similar decomposition into $\sqcup_{i \in C_G} M^G_i$

\begin{prop}\label{prop:Reeb_MB}
The Reeb flow is Morse-Bott non-degenerate with the following properties:
\begin{enumerate}
    \item The Morse-Bott families of Reeb orbits are parameterized by $G\in S$, $k\in G\backslash G^-$ and $\ell \in \bZ_{\ge 0}$, where the family of parameterized Reeb orbit is diffeomorphic to $M^G$ via the starting point of the Reeb orbit and the period of the orbit is $\ell+\frac{k}{|G|}$ (we parameterize $S^1$ by $[0,1]$). The corresponding Morse-Bott family of unparameterized Reeb orbits is diffeomorphic to the orbifold $M^G/S^1$, which is a branch locus $\cY_G$ of the base orbifold $\cY$ with isotropy $G$. We will call $(G,k,\ell,i)$ the signature of such a Reeb orbit in the component $M^G_i$, where $i\in C_G$. When $G=\{e\}$, we have $\ell\ge 1$, corresponding to $\ell$th cover of a principal orbit.
    \item The lower Conley-Zehnder index $\mu_{\LCZ}(G,k,\ell,i)$  and  Robbin-Salamon index $\mu_{\RS}(G,k,\ell,i)$  of orbits with signature $(G,k,\ell,i)$ satisfy the following properties:
    \begin{enumerate}
        \item $\mu_{\LCZ}(G,k,\ell,i)=\mu_{\RS}(G,k,\ell,i)-\dim_{\bC}{\cY^i_G}$;
        \item $\mu_{\RS}(\{e\},0,\ell)=2\ell R(M)$ (there is only one component for $\pi_0(\cY_{\{e\}}=\cY)$);
        \item $\mu_{\RS}(G,k,\ell,i) = \mu_\RS(G,k,0,i)+2\ell R(M)$;
        \item\label{p:d} $\mu^{\bZ_2}_{\LCZ}(G,k,\ell,i)=  n-1$. 
    \end{enumerate}
    \item If $R(M)\ge 0$, then 
    $$\inf_{\gamma}(\lSFT(\gamma)) = \min_{G,k,i} \{ \mu_{\RS}(G,k,0,i)-\dim_{\bC}\cY^i_G+n-3\}\cup \{2R(M)-2\},$$
    where $\gamma$ ranges over Reeb orbits of $\eta$.
\end{enumerate}
\end{prop}
\begin{proof}
    The description of the Morse-Bott family follows from the fact that the Reeb flow is the $S^1$-action. The identity $\mu_{\LCZ}(G,k,\ell,i)=\mu_{\RS}(G,k,\ell,i)-\dim_{\bC}{\cY^i_G}$ follows from \eqref{eqn:LCZ}. Since the Robin-Salamon index for a loop of symplectic matrices is twice the Maslov index, we have $\mu_\RS(\{e\}, 0,\ell)=2\ell R(M)$. The linearized flow for $\mu_\RS(G,k,\ell,i)$ is the linearized flow for $\mu_\RS(G,k,0,i)$ concatenated with $\ell$-multiple of the loop of symplectic matrices from the $S^1$ action, and so $\mu_\RS(G,k,\ell,i) = \mu_\RS(G,k,0,i)+2\ell R(M)$ by the concatenation property. To see 2(d), the mod $2$ lower Conley-Zehnder index is the sum of the mod 2 indices from two components, as the matrix path decomposes into two components: one from $M^G$ and the other from the normal direction. In the direction of $M^G$, the mod $2$ lower Conley-Zehnder index is $\dim_{\bC} \cY^i_G$ as it is a loop of symplectic matrices in $\Sp(\dim_{\bR}\cY^i_G)$. In the normal direction, as the endpoint/linearized return map is complex-linear, we know that the parity of the complex dimension of the normal bundle, i.e.\ $\frac{1}{2}(\dim M-\dim M^G_i)=n-1-\dim_{\bC}\cY^i_G\mod 2$, is the same as the parity of the Conley-Zehnder index in the normal direction, by property \eqref{p:det}, as $\mathrm{sign} \det_{\bR} (\Id-\text{complex linear matrix})=1$ when it (the linearized return map in the normal direction) is non-singular. Then 2(d) follows from property \eqref{p:sum}. For a Reeb orbit of signature $(G, k, \ell, i)$, we have the identity 
    $$\lSFT(\gamma)=\mu_{\RS}(G,k,\ell,i)-\dim_{\bC}\cY^i_G+n-3,$$
    When $R(M)\ge 0$,  2(b) implies that $\inf_{\gamma}\{\lSFT(\gamma)\}$ is attained by Reeb orbits with $\ell = 0$ or $G=\{e\}$ and $\ell=1$, i.e.\ 
    $$\inf_{\gamma}(\lSFT(\gamma)) = \min_{G\ne \{e\},k,i} \{ \mu_{\RS}(G,k,0,i)-\dim_{\bC}\cY^i_G+n-3\}\cup \{2R(M)-2\}.$$
\end{proof}
In the following, we compute directly that $\inf_{\gamma}\lSFT(\gamma)$ is twice the minimal discrepancy in \eqref{eqn:md} for the conic contact form $\eta$. 
We will give two proofs of \eqref{eq-lSFT}. The one given below uses axiomatic properties of the Conley-Zehnder index. Another proof is contained in Appendix \ref{append}. 

\begin{prop}\label{prop:CZ_cone}
    Let $o\in \cC$ be an isolated Fano cone singularity of dimension $n$.  For any conic contact form $\eta$, we have
    \begin{equation}\label{eq-lSFT}
        \inf_{\gamma} \lSFT(\gamma) = 2\,\mathrm{md}(o,\cC).
    \end{equation}
\end{prop}

\begin{proof}
    Locally, the holomorphic vector bundle is modeled on $\bC\times \bC^{n-1}$ with a cyclic action $\frac{1}{m}(w_1,\ldots,w_n)$, where the first factor is the fiber direction and $(m,w_1)=1$ as the link is smooth. The contact form on the associated $S^1$ bundle, i.e.\ $M$, is given by a connection form whose curvature is $\sqrt{-1}$ times the symplectic form on the base $\cY$. The choice of connections does not matter locally, as they differ by a gauge transformation. In the tangent space of $0$ in the base $\bC^{n-1}$, complex subspaces are symplectic, and the symplectic complement of an invariant symplectic subspace is also invariant, as the symplectic form is invariant. Hence there exists a symplectic basis such that the tangent space at $0$, as a symplectic vector space, is $(\bC^{n-1},\omega_{std})$ with a cyclic quotient action $\frac{1}{m}(w_2,\ldots,w_n)$. By Moser's trick, one can argue that the symplectic form on the base is locally modeled on $(\bC^{n-1},\omega_{std})$ with a cyclic quotient action $\frac{1}{m}(w_2,\ldots,w_n)$. As $\omega_{std}=\rd \lambda_{std} = \rd(\frac{1}{2}\sum_{i=2}^n (x_{i}\rd y_i-y_i\rd x_i))$ and $\lambda_{std}$ is invariant under the cyclic action, we can use $\rd \theta+\lambda_{std}$ as the connection on the trivialization. Therefore the contact form  is locally modeled on $(S^1\times \bC^{n-1}, \rd \theta+\lambda_{std})$ with a cyclic action action $\frac{1}{m}(w_1,\ldots,w_n)$. Note that we are only claiming the contact form, not the Sasaki structure, is standard, as Moser's trick only preserves the symplectic structure. With such a standard contact model, the Reeb flow is given by $\partial_{\theta}$. The projection $\pi:\xi=\ker(\rd \theta+\lambda_{std})\subset T(S^1\times \bC^{n-1})\to \bC^{n-1}$ is an isomorphism, and $\pi_*(\rd (\theta+\lambda_{std})) = \rd \lambda_{std}=\omega_{std}$ on $\bC^{n-1}$. In order to make the principal orbits have period $1$ as in \Cref{prop:Reeb_MB}, we rescale the contact form to $\frac{1}{2\pi}(\rd \theta + \lambda_{std})$, such that the Reeb flow is $\phi_t(\theta, v)=(\theta+2\pi t, v)$. Let $\gamma$ denote the Reeb orbit $(e^{2\pi w_1\sqrt{-1}t/m},0)$ for $t\in [0,1]$, we consider the symplectic trivialization $\gamma^*\xi$ by 
    $$S^1\times \bC^{n-1} \to \gamma^*\xi$$
    $$(t,v=(v_2,\ldots,v_n))\mapsto \pi^{-1}(e^{2\pi w_2\sqrt{-1}t/m} v_2,\ldots, e^{2\pi w_n\sqrt{-1}t/m} v_n)$$
    Under such a trivialization $\tau$, the linearized Reeb flow gives rise to the following path of symplectic matrices
    \begin{equation}\label{eqn:linear}
        t\to \mathrm{diag}(e^{-2\pi w_2\sqrt{-1}t/m},\ldots, e^{-2\pi w_n\sqrt{-1}t/m}), \quad t\in [0,1]
    \end{equation}
    With such a trivialization $\tau$, we have
    $$\lSFT^{\tau}(\gamma) = -\# \{ i|i\ge 2, w_i> 0 \} -\# \{ i|i\ge 2, w_i=0\}+n-3=-2,$$
    where $-\# \{ i|i\ge 2, w_i> 0 \}$ is from those axes with $w_i>0$ by Property \eqref{p:sign} of the Conley-Zehnder index, and $-\# \{ i|i\ge 2, w_i=0\}$ is $-\dim_{\bC}\cY_{\bZ/m}=-\frac{1}{2}\dim \ker (\phi(1)-\Id)$ in \eqref{eqn:LCZ}, which is the lower Conley-Zehnder index from those axes with $w_i=0$ by Properties \eqref{p:RS} and \eqref{p:RS2}. Then the claim follows from Property \eqref{p:sum}. Similarly, by Properties \eqref{p:RS} and \eqref{p:RS2}, we have  
    $$\lSFT^{\tau^m}(\gamma^m)= -2\sum_{i=2}^n w_i-2,$$
    where $\tau^m$ is the induced trivialization on $\gamma^m$, under which the linearized flow is \eqref{eqn:linear} for $t\in [0,m]$.
    Since $\gamma^m$ is $w_1$-th cover of a principal orbit, by \Cref{prop:Reeb_MB}, $\lSFT$ using a global trivialization is
    $$\lSFT(\gamma^m) = 2w_1 r(M)-2.$$
    Therefore, we can compute the discrepancy between the trivialization $\tau$ and the global trivialization and conclude that:
    $$\lSFT(\gamma)=\lSFT^{\tau}(\gamma)+\frac{\lSFT(\gamma^m)-\lSFT^{\tau^m}(\gamma^m)}{m}=\frac{2}{m}(\sum^n_{i=1} w_i+(r-1)w_1)-2.$$
    Therefore $\inf_{\gamma}\lSFT(\gamma)$ is twice the minimal discrepancy in \eqref{eqn:md}.
\end{proof}

\subsection{Symplectic homology and its variants}
Let $W$ be the Liouville filling of $M$ in $\bC^n$ in \Cref{prop:Liouville}. To every Liouville filling, one can associate several symplectic invariants, namely the symplectic homology (over $\bQ$) $SH_*(W;\bQ)$ as well as its variants, see e.g.\ \cite{CO}. Those invariants are always graded over $\bZ_2$, and they are generated by (non-degenerate) Reeb orbits/Hamiltonian orbits with grading from the Conley-Zehnder index. When $c_1(W)=0$, those invariants are graded by $\bZ$, where the grading depends on a choice of the trivialization of $\det_{\bC}TW$. If we only have $c^{\bQ}_1(W)=0$, those invariants can be graded over $\bQ$ with similar dependence on trivializations. 

Our proof follows from the computation of the positive $S^1$-equivariant symplectic homology $SH^{+,S^1}_*(W;\bQ)$ of $W$ from two perspectives: (1) from functorial properties of the invariants, (2) direct computation with the help of the $S^1$-symmetry on the Sasaki link and the Morse-Bott spectral sequence \cite{KvK16}.
\begin{rem}
    There are many different conventions regarding symplectic (co)homology and its variants. Since the references cited here are not completely consistent, we point out their relations: the symplectic homology $SH_*$ and its variants in \cite{BO13, CO,KvK16} is isomorphic to the symplectic cohomology $SH^*$ and its variants in \cite{Ri13,Se08} by $SH^*=SH_{n-*}$, where $n$ is half of the dimension of the symplectic domain. The symplectic homology in \cite{BO13, CO,KvK16} is graded by Conley-Zehnder indices of the Hamiltonian orbits.
\end{rem}

\begin{prop}\label{prop:functorial}
    Let $W$ be the strongly pseudo-convex domain bounded by $M$ in $\bC^n$ obtained by \Cref{prop-compactify}, the positive $S^1$-equivariant symplectic homology of $W$ is
    $$SH^{+,S^1}_*(W;\bQ)=\left\{\begin{array}{cc}
        \bQ, &  *=n+1+2m, m\in \bZ_{\ge 0}\\
         0, &  \text{else}
    \end{array}\right..$$
\end{prop}
\begin{proof}
    By gluing the Liouville filling $W$ from \Cref{prop:Liouville} and the Liouville cobordism in \Cref{prop:cobordism}, we get a Liouville filling $\tilde{W}$ of the standard contact sphere. Then by a theorem of Seidel and Smith \cite[Corollary 6.5]{Se08}, the symplectic homology of $\tilde{W}$ vanishes. Because there is a unital ring map $SH_*(\tilde{W};\bQ)\to SH_*(W;\bQ)$ from the Viterbo transfer \cite[Theorem 9.5]{Ri13}, we have the vanishing $SH_*(W;\bQ)=0$. Then by the tautological long exact sequence \cite[Lemma 8.1]{Ri13}
    $$\ldots \to H^*(W;\bQ)\to SH_{n-*}(W;\bQ)\to SH^+_{n-*}(W;\bQ)\to H^{*+1}(W;\bQ)\to \ldots, $$
    we get that $SH^+_*(W;\bQ)$ is supported in degree $n+1$ with rank $1$ as $W$ is a homology ball by \Cref{prop-compactify}. By \cite[Theorem 1.1]{BO13}, we have the following Gysin exact sequence, 
    $$\ldots \to SH^+_k(W;\bQ) \to SH^{+,S^1}_k(W;\bQ)\stackrel{D}{\to} SH^{+,S^1}_{k-2}(W)\to SH^+_{k-1}(W)\to \ldots.$$
    As a consequence, we know that $D: SH^{+,S^1}_k(W;\bQ)\to SH^{+,S^1}_{k-2}(W)$ is an isomorphism when $k\ge n+3$ or $k\le n$. The map $D$ is nilpotent in the sense that for every $x\in SH^{+,S^1}_*(W;\bQ)$ there exists a $K\in \bZ{\ge 0}$ depending on $x$, such that $D^kx=0$. This follows from the definition of $D$ \cite{BO13} as it decreases the action.  Therefore we have $SH^{+,S^1}_k(W;\bQ)=0$ when $k\le n$. Then from the Gysin exact sequence, we have exact sequences:
       $$\ldots \to 0 \to SH^+_{n+1}(W;\bQ) \to SH^{+,S^1}_{n+1}(W;\bQ)\stackrel{D}{\to} 0 \to \ldots,$$
    and
     $$\ldots \to 0 \to SH^{+,S^1}_{n+2}(W;\bQ)\stackrel{D}{\to} 0 \to \ldots.$$
    So we get $SH^{+,S^1}_{n+1}(W;\bQ)=\bQ$ and  $SH^{+,S^1}_{n+2}(W;\bQ)=0$. The claim now follows because we know that $D$ is an isomorphism when $k\ge n+3$.
\end{proof}

The next result works for more general $M$, namely for those \Cref{prop:Reeb_MB} applies. Let $W$ be a Liouville symplectic filling of $M$ such that the rational first Chern class $c^{\bQ}_1(W)=0$. We have a Morse-Bott spectral sequence \cite[Theorem 5.4. (5.3)]{KvK16} computing $SH^{+,S^1}_*(W;\bQ)$. Its first page is given by:
$$E_{p,q}^1=\bigoplus_{p=N(\ell + \frac{k}{|G|})} H_{p+q-\mu_{\LCZ}(G,k,\ell,i)}(\cY^i_G;\bQ)$$
where $N=\mathrm{lcm}_{G\in S}(|G|)$. Here, the spectral sequence is from the filtration by the period recorded by $p$ and $p+q$ is the (rational) grading on $SH^{+,S^1}_*(W;\bQ)$ by the (rational) Conley-Zehnder indices. This spectral sequence is induced by the \emph{increasing} filtration by subcomplexes generated by Reeb orbits with period up to $p/N$. This is a \emph{homological} spectral sequence with differential $d^r_{p,q}:E^r_{p,q}\to E^r_{p-r,q+r-1}$. Moreover, each element from $H_{*}(\cY^i_G;\bQ)$ also has a well-defined $\bZ_2$ grading corresponding to the \emph{same} spectral sequence from the period filtration but graded by the mod $2$ Conley-Zehnder index. The differential $d^r$ is degree $-1$ for both the $\bQ$-grading and $\bZ_2$-grading. In the following picture, we spell out the $\bQ$-grading and use different colors to mark the $\bZ_2$-grading of each element.
\begin{rem}
    \cite[Theorem 5.4, condition (ii)]{KvK16} requires that $c_1(\xi),\pi_1(M)=0$. Such a condition is only needed to obtain a canonical $\bZ$-grading on symplectic homology. Moreover, the actual condition needed to obtain a canonical $\bZ$-grading is that all Reeb orbits are null-homologous and $c_1(\xi)=0$. While the spectral sequence is induced from a period/action filtration, which always exists in a Morse-Bott situation by the proof of \cite[Theorem 5.4]{KvK16}. The extra condition equips the spectral sequence with $\bZ$-grading.  In our case here, we have $c_1^{\bQ}(\xi)=0$ and all Reeb orbits have torsion homology class, this equips symplectic homology and the spectral sequence with a canonical $\bQ$-grading using the rational Conley-Zehnder index recalled in \S \ref{ss:CZ}.
\end{rem}

    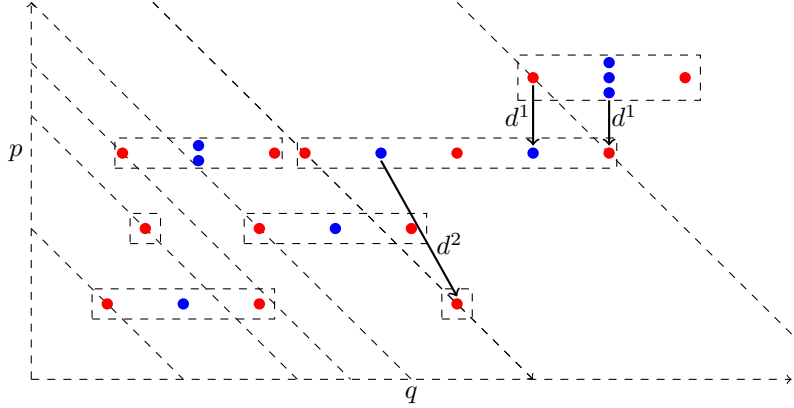
\begin{figure}[H]
       \begin{center}
       \begin{tikzpicture}
        \draw[dashed] (0,2) to (2,0);
        \draw[dashed] (0,3.5) to (3.5,0);
        \draw[dashed] (0,5) to (5,0);
        \draw[dashed] (0,4.2) to (4.2,0);
        \draw[dashed] (1.6,5) to (6.6,0);
        \draw[dashed] (10,0.6) to (5.6,5);
        \draw[dashed,->] (0,0) to (10,0);
        \draw[dashed,->] (0,0) to (0,5);
        \draw[->,thick] (7.6,3.7) to (7.6,3.1);
        \draw[->,thick] (6.6,3.9) to (6.6,3.1);
        \draw[->,thick] (4.6,2.9) to (5.6,1.1);
        \draw[->,dashed] (1.6,5) to (6.6,0);
        \node at (7.8,3.5) {$d^1$};
        \node at (6.4,3.5) {$d^1$};
        \node at (5.5,1.8) {$d^2$};

        \node at (1,1) [circle,fill,inner sep=1.5pt,color=red] {};
        \node at (2,1) [circle,fill,inner sep=1.5pt,color=blue] {};
        \node at (3,1) [circle,fill,inner sep=1.5pt, color=red] {};
        \draw[dashed] (0.8,0.8) to (3.2,0.8) to (3.2,1.2) to (0.8,1.2) to (0.8,0.8);

        \node at (5.6,1) [circle,fill,inner sep=1.5pt,color=red] {};
        
        \draw[dashed] (5.4,0.8) to (5.8,0.8) to (5.8,1.2) to (5.4,1.2) to (5.4,0.8); 
        
        \node at (1.5,2) [circle,fill,inner sep=1.5pt,color=red] {};
        \draw[dashed] (1.3,1.8) to (1.7,1.8) to (1.7,2.2) to (1.3,2.2) to (1.3,1.8);
        \node at (3,2) [circle,fill,inner sep=1.5pt,color=red] {};
        \node at (4,2) [circle,fill,inner sep=1.5pt,color=blue] {};
        \node at (5,2) [circle,fill,inner sep=1.5pt,color=red] {};
        \draw[dashed] (2.8,1.8) to (5.2,1.8) to (5.2,2.2) to (2.8,2.2) to (2.8,1.8);
        
        \node at (1.2,3) [circle,fill,inner sep=1.5pt,color=red] {};
        \node at (2.2,3.1) [circle,fill,inner sep=1.5pt,color=blue] {};
        \node at (2.2,2.9) [circle,fill,inner sep=1.5pt,color=blue] {};
        \node at (3.2,3) [circle,fill,inner sep=1.5pt,color=red] {};
        \draw[dashed] (1.1,2.8) to (3.3,2.8) to (3.3,3.2) to (1.1,3.2) to (1.1,2.8);

        \node at (3.6,3) [circle,fill,inner sep=1.5pt,color=red] {};
        \node at (4.6,3) [circle,fill,inner sep=1.5pt, color=blue] {};
        \node at (5.6,3) [circle,fill,inner sep=1.5pt,color=red] {};
        \node at (6.6,3) [circle,fill,inner sep=1.5pt,color=blue] {};
        \node at (7.6,3) [circle,fill,inner sep=1.5pt,color=red] {};
        \draw[dashed] (3.5,2.8) to (7.7,2.8) to (7.7,3.2) to (3.5,3.2) to (3.5,2.8);

        \node at (6.6,4) [circle,fill,inner sep=1.5pt,color=red] {};
        \node at (7.6,4.2) [circle,fill,inner sep=1.5pt,color=blue] {};
        \node at (7.6,4) [circle,fill,inner sep=1.5pt,color=blue] {};
        \node at (7.6,3.8) [circle,fill,inner sep=1.5pt,color=blue] {};
        \node at (8.6,4) [circle,fill,inner sep=1.5pt,color=red] {};
        \draw[dashed] (6.4,3.7) to (8.8,3.7) to (8.8,4.3) to (6.4,4.3) to (6.4,3.7);

        \node at (-0.2,3) {$p$};
        \node at (5,-0.2) {$q$};

		\end{tikzpicture}
        \caption{A schematic picture of the first page of the spectral sequence}
    \end{center}
    \end{figure}
    To present the first page of the spectral sequence, the vertical axis is $p$ and the horizontal axis is $q$; each dot represents a copy of $\bQ$ with color indicating their $\bZ_2$ grading. Each dashed block represents some $H_*(\cY^i_G;\bQ)$ with the $p+q$ coordinate of the leading term is $\mu_{\LCZ}(G,k,\ell,i)$. All the leading dots have the same $\bZ_2$ grading by \eqref{p:d} of \Cref{prop:Reeb_MB}, where red means $n-1 \mod 2$.
    The arrows are differentials that respect both gradings.

\begin{prop}\label{prop:hmi}
     Under the assumption of \Cref{prop:Reeb_MB} and $R(M)>0$. If there is a Liouville filling $W$ of $M$ such that $c_1^{\bQ}(W)=0$, 
     we have 
     $$\inf_{\gamma}(\lSFT(\gamma))=\inf\{k| SH^{+,S^1}_k(W;\bQ)\ne 0\}+n-3 =\hmi(M),$$ 
     where $\gamma$ ranges over Reeb orbits of a conic contact form $\eta$.
\end{prop}
\begin{proof}
    We know that $SH^{+,S^1}_*(W;\bQ)$ is an invariant independent of $\eta$ whose underlying cochain complex is supported in degrees at least $\inf_{\gamma}(\lSFT(\gamma))+3-n$ for any contact form $\eta$. Therefore for any contact form $\eta$, we have 
    $$\inf_{\gamma}(\lSFT(\gamma))\le \inf\{k| SH^{+,S^1}_k(W;\bQ)\ne 0\}+n-3$$
    where we enumerate $\gamma$ through all Reeb orbits of $\eta$. Hence, we know that 
    $$\hmi(Y) \le \inf\{k| SH^{+,S^1}_k(W;\bQ)\ne 0\}+n-3.$$ 
    Now we return to the case of conic contact forms. As $R(M)>0$, $\inf_{\gamma}(\lSFT(\gamma)) = \min_{G,k} \{ \mu_\RS(G,k,0,i)-\dim_{\bC}\cY^i_G+n-3\}\cup \{2R(M)-2\}$. We claim that an element in $H_0(\cY^i_G;\bQ)$ with the minimal rational degree $p+q$ and the maximal $p$ will survive in the spectral sequence. The $R(M)>0$ condition implies that concatenation with principal orbits increases both the period and grading, hence such elements exist. For such an element $\alpha$ to be killed in the spectral sequence, we must have an element $\beta$ in $E^1_{p+r,q-r+1}$ that survives till the $r$th page and $d^r\beta=\alpha$. 
    
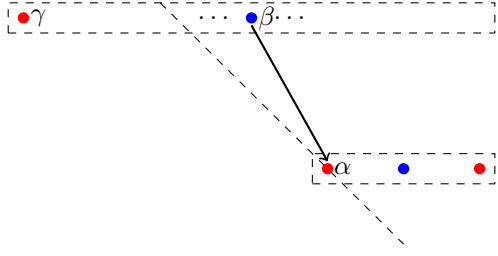
      \begin{figure}[H]
       \begin{center}
       \begin{tikzpicture}
        \draw[dashed] (-2.2,2.2) to (1,-1);
        \node at (0,0) [circle,fill,inner sep=1.5pt,color=red] {};
        \node at (1,0) [circle,fill,inner sep=1.5pt,color=blue] {};
        \node at (2,0) [circle,fill,inner sep=1.5pt, color=red] {};
        \draw[dashed] (-.2,-.2) to (2.2,-.2) to (2.2,.2) to (-.2,.2) to (-.2,-.2);

        \node at (-1,2) [circle,fill,inner sep=1.5pt,color=blue] {};
        \node at (-4,2) [circle,fill,inner sep=1.5pt,color=red] {};
        \node at (-1.5,2)  {$\ldots$};
        \node at (-.5,2)  {$\ldots$};
        \draw[->,thick] (-1,1.9) to (0,0.1);
        \draw[dashed] (-4.2,1.8) to (2.2,1.8) to (2.2,2.2) to (-4.2,2.2) to (-4.2,1.8);
        \node at (0.2,0) {$\alpha$};
        \node at (-0.8,2) {$\beta$};
        \node at (-3.8,2) {$\gamma$};
       \end{tikzpicture}
       \end{center}
       \end{figure}
    \noindent
    Since the $\alpha$ and $\beta$ have different $\bZ_2$ grading, we must have $\alpha$ red and $\beta$ blue, i.e.\ in the situation shown in the picture above. As a consequence, the minimal degree term $\gamma$ of the block containing $\beta$ must have strictly larger $p$ and no larger $p+q$ compared to $\alpha$. But this contradicts with the choice of $\alpha$. Therefore $\alpha$ is non-trivial in $SH_*^{+,S^1}(W;\bQ)$ with the minimal grading $\inf_{\gamma}(\lSFT(\gamma))+3-n$. On the other hand, we have $\hmi(M)\ge \inf_{\gamma}(\lSFT(\gamma))$ by definition. So we must achieve the following equalities:
    $$\inf_{\gamma}(\lSFT(\gamma))=\inf\{k| SH_k^{+,S^1}(W;\bQ)\ne 0\}+n-3=\hmi(M). $$
\end{proof}
\begin{rem}\label{rmk:aug}
     Assume the Reeb flow of $\eta$ gives an $S^1$-action as above and $R(M)>0$. It is always true that $\hmi(M)$ is the $\inf_{\gamma}(\lSFT(\gamma))$, which is observed by a homology class, i.e.\ we do not need to assume the existence of a Liouville filling $W$. The existence of a Liouville filling allows us to use $\bQ$-graded positive $S^1$-equivariant symplectic homology, which is the linearized contact homology w.r.t.\ the augmentation from $W$ \cite{BO09}. Such a Liouville filling always exists in our setting if $o\in \cC$ is smoothable and $\dim_{\bC} \cC\ge 3$, or in the compactification problem by \Cref{prop:Liouville}. In general, we can use the linearized contact homology w.r.t.\ an algebraic $\bQ$-graded augmentation over $\bC$. The Morse-Bott spectral sequence works for such linearized contact homology as well, and the proof of \Cref{prop:hmi} holds. As contact manifolds with $S^1$-Reeb flow are strongly fillable by \cite{KL21}, the contact homology does not vanish \cite[Theorem 1 and Proposition 2.9]{LW11}. Since the vanishing of contact homology is the only obstruction to the existence of an algebraic augmentation over an algebraically closed field, we can always use the linearized contact homology, which is expected to be a contact invariant\footnote{The foundation of linearized contact homology is an ongoing project \cite{DG} of Russell Avdek, Erkao Bao, Georgios Dimitroglou Rizell, and the second author.}.
\end{rem}

\subsection{Proofs of the main theorems}

\begin{proof}[Proof of Theorem \ref{thm:CZ}]
The statement follows immediately by combining \eqref{eqn:md}, Propositions \ref{prop:CZ_cone} and Proposition \ref{prop:hmi}.
\end{proof}

\begin{proof}[Proof of Theorem \ref{thm-AC}]
By approximation, we can assume that there exists a sequence of quasi-regular $\RV_k$ that converges to $\RV$ as $k\rightarrow +\infty$. For each $\RV_k$, we apply Theorem \ref{thm-compactify} to get K\"{a}hler orbifold compactification of $\bC^n$. By Proposition \ref{prop-klt} we know that $\cC$ has a Gorenstein isolated singularity at the vertex $o$.  

By Propositions \ref{prop-compactify} and Proposition \ref{prop:Fano_Maslov}, we have $R(M)>0$. Then Propositions \ref{prop:hmi} and Proposition \ref{prop:functorial} imply that $\inf_{\gamma}(\lSFT(\gamma))=\inf\{k| SH^{+,S^1}_k(W;\bQ)\ne 0\}+n-3=2n-2$, which implies the minimal discrepancy number is $n-1$ by \Cref{thm:CZ}.

If we assume that Shokurov's conjecture is true, in particular if the dimension $n\le 3$, then $\cC$ is a smooth affine variety with an effective torus action with a unique fixed attractive point $o\in \cC$. By \cite{KR82}, we know that $\cC$ is equivariant biholomorphic to $\bC^n$ with a linear torus action. So the Reeb vector field of $(\cC, g_0)\cong (\bC^n, g_0)$ is exactly given by the imaginary part of some linear vector field $\sum_{i=1}^n w_i z_i \partial_{z_i}$ with $w_i>0, i=1,\dots, n$. Now if $g$ is Ricci flat, then $(\cC, g_0)$ is also Ricci flat. By the volume minimization property of Reeb vector fields associated to Ricci flat K\"{a}hler cone metrics (see \cite{MSY08}), the vector $(w_1, \dots, w_n)$ must minimize the normalized functional $\frac{(w_1+\cdots+w_n)^n}{w_1\cdot\cdots\cdot w_n}$. This forces $w_1=w_2=\cdots=w_n$ and $g_0$ is the flat metric on $\bC^n$. In other words, the metric tangent cone of $(\bC^n, g)$ at infinity is the flat metric on $\bC^n$. By the rigidity result from \cite{And90}, we know that the Ricci flat metric $g$ on $\bC^n$ itself must be flat. 
\end{proof}

\begin{proof}[Proof of \Cref{thm-main1}]
    By \cite[Theorem 1.1.III]{BM78} or Proposition \ref{prop-compactify}, we know that $X$ and $Y$ are Fano manifolds satisfying $-K_X=(r+1)[Y]$ for $r>0$. According to the construction of Tian-Yau (see Example \ref{exmp-TY}), there exists a complete AC metric on $X\setminus Y$ whose asymptotical cone is $\cC(Y, L)$ with $L=[Y]|_Y$.  So we can apply Theorem \ref{thm-AC} to the case when $(\cY, \CL)=(Y, L)$ has no orbifold singularities. Since $-K_Y=r L$, by \eqref{eqn:md} it is easy to see that $\mathrm{md}(o, \cC(Y, L))+1$ is now equal to $r$. So by Theorem \ref{thm-AC}, $r$ is equal to $n$. By Kobayashi-Ochiai's criterion (\cite{KO73}), we know that $Y$ is holomorphic to $\bP^{n-1}$ and the conclusion now follows from the results in \cite[Theorem 1.1.(II)]{BM78}.
\end{proof}

\appendix
\section{An alternative proof of Proposition \ref{prop:CZ_cone} }\label{append}

We will use the same notation as our proof of the formula \eqref{eqn:md}. 
Let $\cC=\cC(\cY, \CL)$ be an isolated Fano cone singularity. Let $\cC'$ denote the total space of the orbifold line bundle $\pi: \CL^{-1}\rightarrow \cY$ and $\mu_1: \cC'\rightarrow \cC$ be the contraction of the zero section. Fix any point $p\in \cY$, by \cite[25]{Kol04} there exist a neighborhood $U$ of $p\in \cY$ and $m\in \bZ_{>0}$ and $e_2, \dots, e_n\in \{0, 1,\dots, m-1\}$ such that
\begin{equation*}
\pi^{-1}(U)=\bC\times \bC^{n-1}/\frac{1}{m}(1, e_2,\dots, e_{n}). 
\end{equation*} 
Let $\{x_1, x_2,\dots, x_{n}\}$ denote the coordinates on the ramified cover $\bC^n$. 
We have the isomorphism
\begin{equation*}
\bC^*\times \bC^{n-1}/\frac{1}{m}(1, e_2, \dots, e_n)=\mathrm{Spec}(\bC[x_1^{m}, x_2 x_1^{-e_2}, \dots, x_n x_1^{-e_n}]).
\end{equation*}
Assume that $-K_{\cY}=r \CL$ so that $\cC$ is $\bQ$-Gorenstein, i.e.\ there exists $N\in \bZ_{>0}$ such that $N \cdot K_{\cC}$ is Cartier. Then we can write down a local trivializing section of $(K_{\cC})^{\otimes N}$ as $${\sigma}_{\cC}=\left(x_1^{r-1} dx_1\wedge dx_2\wedge \cdots \wedge dx_n\right)^{\otimes N}$$ which induces a trivializing section of $((\det \xi)^\vee)^{\otimes N}$ as ${\sigma}:=(\iota_{x_1\partial_{x_1}}\sigma_{\cC})^{\otimes N}= (x_1^r dx_2\wedge \cdots \wedge dx_n)^{\otimes N}$. 
Choose an orbifold Hermitian metric $h$ on the orbifold line bundle $\CL^{-1}\rightarrow \cY$ such that its Chern curvature $\sqrt{-1}\partial\bar{\partial}\log h$ is an orbifold K\"{a}hler form. Locally over the uniformization chart $h$ is represented as $a(x')|x_1|^2$ where $x'=\{x_2,\dots, x_n\}$. For any $\beta>0$, $h^\beta$ induces a radius function $\mathfrak{r}$ on $\cC$ which in turn defines a Riemannian cone metric $g_\cC$ on $\cC$. Over the regular part $\cC^\circ$, we can find a basis of the contact distribution $\xi$ that are compatible with the trivialization of $(\det \xi)^{\otimes N}$:
\begin{equation*}
V_j:=x_1^{e_j}(\partial_{x_j}-a^{-1} (\partial_{x_j}a) x_1\partial_{x_1}), \quad 2\le j\le n. 
\end{equation*}
Indeed, it is easy to verify that $\{V_j: 2\le j\le n\}$ are a basis for the orthogonal complement to the Euler vector field $x_1\partial_{x_1}$ and are invariant under the cyclic group action. 
Let $\nabla$ denote the Levi-Civita connection of $g_\cC|_{\{\mathfrak{r}=1\}}$. Set $x_1=|x_1| e^{\sqrt{-1}\theta}$. By a straightforward calculation, we get:
\begin{equation*}
\nabla_{\partial_\theta}V_j= -(\beta+e_j)V_j, \quad \nabla_{\partial_\theta} \sigma=N(r-\beta (n-1))\sigma. 
\end{equation*}
So if we choose $\beta=\frac{r}{n-1}$ then the parallel transport trivializes $(\det \xi)^{\otimes N}$. From now on, we set $\beta$ to be of this value. 
Let $\gamma$ be the parametrized closed Reeb orbit: $\theta\rightarrow [(\theta, 0,\dots, 0)]$ with $\theta\in [0, \frac{2\pi}{m}]$. We now write down a trivialization of $\gamma^*\xi^{\oplus N}$ by an appropriate choice of basis that is compatible with the trivialization of $\det(\xi)^{\otimes N}$.
First, we consider the Gorenstein case, i.e.\ when $N=1$. We set
\begin{equation*}
V'_2=e^{\sqrt{-1}(\beta (n-1)+\sum_{j=2}^n e_j)\theta}V_2,  \qquad V'_j=V_j, \quad 3\le j\le n.
\end{equation*}
Note that we have the Lie derivatives:
\begin{equation*}
\mathscr{L}_{\partial_\theta}V'_2=\sqrt{-1}\left((n-1)\beta+\sum_{j=2}^n e_j-e_2\right)V'_2, \quad \mathscr{L}_{\partial_\theta}V'_j=-\sqrt{-1} e_j V'_j. 
\end{equation*}
So the linearized flow $\{\phi_\theta: 0\le \theta\le 2\pi/m\}$ is represented by the diagonal matrix:
\begin{equation*}
(e^{\sqrt{-1}((n-1)\beta+\sum_{j=2}^n e_j)\theta}, 1,\dots, 1)\cdot (e^{-\sqrt{-1}e_2\theta}, \dots, e^{-\sqrt{-1}e_{n} \theta}).
\end{equation*}
Since $(n-1)\beta=r$, its CZ index is equal to 
\begin{equation*}
2\frac{1}{m}(r+\sum_{j=2}^n e_j)-\sharp \{j: e_j>0\}=2 \frac{1}{m} (r w_1+\sum_{j=2}^n w_j)-(n-1)+\frac{1}{2}\dim \mathrm{ker}(\phi_{2\pi/m}-\mathrm{Id})
\end{equation*}
where $w_1=1, w_j=e_j, 2\le j\le n$. 
Similarly for $1\le k<m$, the CZ index of $\gamma^k$ is calculated by using the path $\{\phi_\theta, 0\le \theta\le 2\pi\frac{k}{m}\}$ (see \eqref{eq-RSnorm}):
\begin{eqnarray*}
&&2\frac{k}{m}\left(r+\sum_{j=2}^n e_j\right)-\sum_{j: e_j>0}\left(2\left\lfloor \frac{e_j k}{m}\right\rfloor+1\right)=2 \frac{1}{m}\left(k r+\sum_{j: e_j>0}(k e_j\; \mathrm{mod}\, m)\right)-\sharp\{j: e_j>0\}\\
&=&2 \frac{1}{m}(rw_1+\sum_{j=2}^n w_j)-(n-1)+\frac{1}{2}\dim \mathrm{ker}(\phi_{2\pi k/m}-\mathrm{Id}) 
\end{eqnarray*}
where we used the identity $(w_1, w_2, \dots, w_n)=(k, k e_2\, \mathrm{mod}\, m, \dots, k e_n\, \mathrm{mod}\, m)$. When $k=m$, $\lSFT(\gamma^m)$ is $2r-2$.
So we indeed get:
\begin{equation}\label{eq-lSFTapp}
\inf_{1\le k\le m} \mathrm{lSFT}(\gamma^k)=2\min\left\{r,  \frac{1}{m}(r w_1(g)+\sum_{j=2}^n w_j(g)), g\neq \mathrm{id}\in G\right\}-2. 
\end{equation}
When $m > k$, $\lSFT$ can be computed similarly but will be strictly larger than \eqref{eq-lSFTapp} as $r>0$.
Finally for the case when $N\ge 2$, we can trivialize $\gamma^*\xi^{\oplus N}$ by using the following basis vectors:
\begin{equation*}
V'^{(p)}_j=\left\{
\begin{array}{ll}
  e^{\sqrt{-1} N(\beta (n-1)+\sum_{j=2}^n e_j)\theta}V_2,   & \text{if}\;  p=1, j=2 \\
  V_j,   &  \text{otherwise i.e. if } p=1, 3\le j\le n \text{\; or\; } 2 \le p\le N, 2\le j\le n.
\end{array} 
\right.
\end{equation*}
Then a similar calculation gives the same expression as \eqref{eq-lSFTapp}.

\vskip 5mm 
\noindent \textbf{Postscript Note:}
After this paper was posted on arXiv,  Prof.\ Thomas Peternell posted a preprint \cite{Pet24} about a completely different approach (though a direct computation of Chern numbers) to Theorem \ref{thm-main1}. 
In \cite{PL25}, the authors used such an approach to give a proof of Theorem \ref{thm-main1} when the dimension $n\not\equiv 3 \; (\mathrm{mod}\; 4)$.

\bibliographystyle{alpha} 
\bibliography{ref}

\begin{thebibliography}{BdFFU14}

\bibitem[ABDRZ]{DG}
Russell Avdek, Erkao Bao, Georgios Dimitroglou~Rizell, and Zhengyi Zhou.
\newblock Chain level contact homology.
\newblock {\em in preparation}.

\bibitem[And90]{And90}
Michael~T. Anderson.
\newblock Convergence and rigidity of manifolds under ricci curvature bounds.
\newblock {\em Invent. Math.}, 102(2):429–445, 1990.

\bibitem[BdFFU14]{BdFFU}
S.~Boucksom, T.~de~Fernex, C.~Favre, and S.~Urbinati.
\newblock Valuation spaces and multiplier ideals on singular varieties.
\newblock In {\em Recent advances in algebraic geometry. A volume in honor of Rob Lazarsfeld's 60th birthday. Based on the conference, Ann Arbor, MI, USA, May 16--19, 2013}, pages 29--51. Cambridge: Cambridge University Press, 2014.

\bibitem[BG08]{BG08}
Charles Boyer and Krzysztof Galicki.
\newblock {\em Sasakian geometry}.
\newblock Oxford Univ. Press, 2008.

\bibitem[BM78]{BM78}
L.~Brenton and J.~Morrow.
\newblock Compactifications of {${\bf C}\sp{n}$}.
\newblock {\em Trans. Amer. Math. Soc.}, 246:139--153, 1978.

\bibitem[BO09]{BO09}
Fr\'{e}d\'{e}ric Bourgeois and Alexandru Oancea.
\newblock An exact sequence for contact- and symplectic homology.
\newblock {\em Invent. Math.}, 175(3):611--680, 2009.

\bibitem[BO13]{BO13}
Fr\'{e}d\'{e}ric Bourgeois and Alexandru Oancea.
\newblock The {G}ysin exact sequence for {$S^1$}-equivariant symplectic homology.
\newblock {\em J. Topol. Anal.}, 5(4):361--407, 2013.

\bibitem[CH24]{CH24}
Ronan~J. Conlon and Hans-Joachim Hein.
\newblock Classification of asymptotically conical calabi-yau manifolds.
\newblock {\em Duke Math. J.}, 173(5):947--1015, 2024.

\bibitem[CO18]{CO}
Kai Cieliebak and Alexandru Oancea.
\newblock Symplectic homology and the {E}ilenberg-{S}teenrod axioms.
\newblock {\em Algebr. Geom. Topol.}, 18(4):1953--2130, 2018.
\newblock Appendix written jointly with Peter Albers.

\bibitem[CR21]{CR21}
Ronan~J. Conlon and Fr\'{e}d\'{e}ric Rochon.
\newblock New examples of complete calabi-yau metrics on $\mathbb{C}^n$.
\newblock {\em Ann. Sci. \'{E}c. Norm. Sup\'{e}r.}, 54(2):259--303, 2021.

\bibitem[FHYZ24]{FHYZ}
Hanlong Fang, Xiaojun Huang, Wanke Yin, and Zhengyi Zhou.
\newblock Bounding smooth {L}evi-flat hypersurfaces in a {S}tein manifold.
\newblock {\em arXiv preprint arXiv:2409.08470}, 2024.

\bibitem[Gra62]{Gra62}
Hans Grauert.
\newblock \"{U}ber modifikationen und exzeptionelle analytische mengen.
\newblock 146:331--368, 1962.

\bibitem[Gut14]{Gutt14}
Jean Gutt.
\newblock Generalized {C}onley-{Z}ehnder index.
\newblock {\em Ann. Fac. Sci. Toulouse Math. (6)}, 23(4):907--932, 2014.

\bibitem[GZ21]{gironella2021exact}
Fabio Gironella and Zhengyi Zhou.
\newblock Exact orbifold fillings of contact manifolds.
\newblock {\em arXiv preprint arXiv:2108.12247}, 2021.

\bibitem[H\"73]{Ho73}
Lars H\"{o}rmander.
\newblock {\em An introduction to complex analysis in several variables}, volume Vol. 7 of {\em North-Holland Mathematical Library}.
\newblock North-Holland Publishing Co., Amsterdam-London; American Elsevier Publishing Co., Inc., New York, revised edition, 1973.

\bibitem[Hae84]{Hae84}
Andr\'{e} Haefliger.
\newblock Grouo\"{i}des d'holonomie et classifiants.
\newblock {\em Ast\'{e}risque}, (116):70--97, 1984.

\bibitem[Hir54]{Hir54}
Friedrich Hirzebruch.
\newblock Some problems on differentiable and complex manifolds.
\newblock {\em Ann. of Math. (2)}, 60:213--236, 1954.

\bibitem[HS91]{HS91}
Andr\'{e} Haefliger and \'{E}liane Salem.
\newblock Actions of tori on orbifolds.
\newblock {\em Ann. Global Anal. Geom.}, 9(1):37--59, 1991.

\bibitem[Kaw93]{Kaw93}
Yujiro Kawamata.
\newblock The minimal discrepancy of a 3-fold terminal singularity, appendix to v.v.shokurov: 3-fold log flips.
\newblock {\em Russian Acad. Sci. Izv. Math.}, 40:93--202, 1993.

\bibitem[KL21]{KL21}
Marc Kegel and Christian Lange.
\newblock A {B}oothby-{W}ang theorem for {B}esse contact manifolds.
\newblock {\em Arnold Math. J.}, 7(2):225--241, 2021.

\bibitem[KO73]{KO73}
Shoshichi Kobayashi and Takushiro Ochiai.
\newblock Characterizations of complex projective spaces and hyperquadrics.
\newblock {\em J. Math. Kyoto Univ.}, 13:31--47, 1973.

\bibitem[Kol04]{Kol04}
Janos Kollar.
\newblock Seifert ${G}_m$-bundles.
\newblock {\em arXiv:0404386}, 2004.

\bibitem[Kol13]{Kol13}
J.~Koll\'{a}r.
\newblock {\em Singularities of the Minimal Model Program}.
\newblock Cambridge University Press, 2013.
\newblock Cambridge Tracts in Mathematics, Vol. 200.

\bibitem[KR82]{KR82}
T.~Kambayashi and P.~Russell.
\newblock On linearizing algebraic torus actions.
\newblock {\em J. Pure Appl Algebra}, 23:243--250, 1982.

\bibitem[KvK16]{KvK16}
Myeonggi Kwon and Otto van Koert.
\newblock Brieskorn manifolds in contact topology.
\newblock {\em Bull. Lond. Math. Soc.}, 48(2):173--241, 2016.

\bibitem[Li19]{Li19}
Yang Li.
\newblock A new complete calabi-yau metric on $\mathbb{C}^3$.
\newblock {\em Invent. Math.}, 217(1):1--34, 2019.

\bibitem[Li20]{Li20}
Chi Li.
\newblock On sharp rates and analytic compactifications of asymptotically conical {K}\"{a}hler metrics.
\newblock {\em Duke Math. J.}, 169(8):1397--1483, 2020.

\bibitem[LLX20]{LLX20}
Chi Li, Yuchen Liu, and Chenyang Xu.
\newblock A guided tour to normalized volume.
\newblock In {\em Geometry Analysis, In Honor of Gang Tian’s 60th Birthday, Progress in Mathematics, Progress in Mathematics, vol. 333}, pages 167--219. Birkh\"auser/Springer, 2020.

\bibitem[LW11]{LW11}
Janko Latschev and Chris Wendl.
\newblock Algebraic torsion in contact manifolds, with an appendix by {M}ichael {H}utchings.
\newblock {\em Geom. Funct. Anal.}, 21(5):1144--1195, 2011.

\bibitem[Mar96]{Mar96}
Dimitri Markushevich.
\newblock Minimal discrepancy for a terminal cdv singularity is 1.
\newblock {\em J. Math. Sci. Univ.}, 3(2):445--456, 1996.

\bibitem[McL16]{McL16}
Mark McLean.
\newblock Reeb orbits and the minimal discrepancy of an isolated singularity.
\newblock {\em Invent. Math.}, 204(2):505--594, 2016.

\bibitem[MR75]{MR75}
J.~Morrow and H.~Rossi.
\newblock Some theorems of algebraicity for complex spaces.
\newblock 27:167--183, 1975.

\bibitem[MSY08]{MSY08}
Dario Martelli, James Sparks, and Shing-Tau Yau.
\newblock Sasaki-einstein manifolds and volume minimisation.
\newblock 280:611--673, 2008.

\bibitem[Pet24]{Pet24}
Thomas Peternell.
\newblock Compactifications of {$C^n$} and the complex projective space.
\newblock {\em arXiv:2410.07207}, 2024.

\bibitem[PL25]{PL25}
Thomas Peternell and Ping Li.
\newblock Compactification of homology cells, fujita's conjectures and the complex projective space.
\newblock {\em arXiv:2502.01072v2}, 2025.

\bibitem[PS91]{PS91}
Thomas Peternell and Michael Schneider.
\newblock Compactifications of {${\bf C}^n$}: a survey.
\newblock In {\em Several complex variables and complex geometry, {P}art 2 ({S}anta {C}ruz, {CA}, 1989)}, volume 52, Part 2 of {\em Proc. Sympos. Pure Math.}, pages 455--466. Amer. Math. Soc., Providence, RI, 1991.

\bibitem[PS08]{PS08}
Chris~A.M. Peters and Joseph~H.M. Steenbrink.
\newblock {\em Mixed Hodge structures}.
\newblock Springer-Verlag, Berlin, 2008, 2008.

\bibitem[Rit13]{Ri13}
Alexander~F. Ritter.
\newblock Topological quantum field theory structure on symplectic cohomology.
\newblock {\em J. Topol.}, 6(2):391--489, 2013.

\bibitem[Sei08]{Se08}
Paul Seidel.
\newblock A biased view of symplectic cohomology.
\newblock In {\em Current developments in mathematics, 2006}, pages 211--253. Int. Press, Somerville, MA, 2008.

\bibitem[Sze19]{Sze19}
G\'abor Szekelyhidi.
\newblock Degeneration of $\mathbb{C}^n$ and $\text{Calabi-Yau}$ metrics.
\newblock {\em Duke Math. J.}, 168(14):2651--2700, 2019.

\bibitem[Tia06]{Tia06}
Gang Tian.
\newblock Aspects of metric geometry of four manifolds.
\newblock In {\em Inspired by S. S. Chern}, volume~11 of {\em Nankai Tracts Math.}, page 381–397. World Scientific Publishing Co. Pte. Ltd., Hackensack, NJ, 2006.

\bibitem[TY91]{TY91}
Gang Tian and Shing-Tung Yau.
\newblock Complete k\"{a}hler manifolds with zero ricci curvature. ii.
\newblock {\em Invent. Math.}, 106(1):27--60, 1991.

\bibitem[vdV62]{Ven62}
A.~van~de Ven.
\newblock Analytic compactifications of complex homology cells.
\newblock {\em Math. Ann.}, 147:189--204, 1962.

\bibitem[Zho24]{Zh24}
Zhengyi Zhou.
\newblock On fillings of contact links of quotient singularities.
\newblock {\em J. Topol.}, 17, 2024.

\end{thebibliography}

\Addresses
\end{document}